\documentclass[letter,10pt,leqno, english]{amsart}
        \title[Equivariant complex bundles and unitary bordism]
				{Equivariant complex bundles, fixed points and equivariant unitary bordism}
        \author{Andr\'es \'Angel}
      	\address{Departamento de Matem\'aticas\\
          Universidad de los Andes\\
	Carrera 1 N. 18 - 10\\ Bogot\'a, Colombia}
	       \email{ja.angel908@uniandes.edu.co}
	\address{
	 Departamento de Matem\'aticas y Estad\'istica\\
       Universidad del Norte\\ 
       Km. 5 via Puerto Colombia, Barranquilla, Colombia}    
       \email{ajangel@uninorte.edu.co}
         \author{Jos\'e Manuel G\'omez}
          \address{ Escuela de Matem\'aticas\\
          Universidad Nacional de Colombia sede Medell\'in\\
	Calle 59A No 63-20, Bloque 43, Oficina 243\\ Medell\'in, Colombia }      
       \email{jmgomez0@unal.edu.co}
    \urladdr{https://sites.google.com/a/unal.edu.co/jmgomez0/}
      \author{Bernardo Uribe}
      \address{Departamento de Matem\'aticas y Estad\'istica\\
       Universidad del Norte\\ 
       Km. 5 via Puerto Colombia, Barranquilla, Colombia}
      \email{bjongbloed@uninorte.edu.co}
      \urladdr{https://sites.google.com/site/bernardouribejongbloed/}
\thanks{The authors acknowledge and thank the financial support 
provided by the Max Planck Institute for Mathematics and the Office of External Activities of the ICTP through Network NT8. The first author acknowledges 
and thanks the financial support of the  grant \textit{P12.160422.004/01- FAPA Andres Angel} 
from Vicedecanatura de Investigaciones de la Facultad de Ciencias de la Universidad de
los Andes, Colombia. The last two authors
acknowledge and thank the support of  COLCIENCIAS through grant numbers FP44842-617-2014  
and FP44842-013-2018 of the Fondo Nacional de Financiamiento para la Ciencia, la Tecnolog\'ia y la Innovaci\'on.
The third author acknowledges and thanks the financial support provided by the 
Alexander Von Humboldt Foundation.}
\date{\today}

\usepackage{pdfsync}
\usepackage{hyperref}
\usepackage{calc}
\usepackage{enumerate,amssymb,amscd}
\usepackage{color}
\usepackage[arrow,curve,matrix,tips,2cell]{xy}
  \SelectTips{eu}{10} \UseTips
  \UseAllTwocells
\usepackage{tikz}
\usepackage{pdfcolmk}

\DeclareMathAlphabet{\matheurm}{U}{eur}{m}{n}

\DeclareMathOperator{\Ad}{Ad}

\DeclareMathOperator{\Inn}{Inn}

\DeclareMathOperator{\map}{map}

\DeclareMathOperator{\Hom}{\textup{Hom}}
\DeclareMathOperator{\Irr}{\textup{Irr}}

  \newcommand{\IC}{\mathbb{C}}

  \newcommand{\IN}{\mathbb{N}}

  \newcommand{\IR}{\mathbb{R}}
  \newcommand{\IS}{\mathbb{S}}

  \newcommand{\IV}{\mathbb{V}}

  \newcommand{\IZ}{\mathbb{Z}}

  \newcommand{\cala}{\mathcal{A}}

  \newcommand{\calf}{\mathcal{F}}

  \newcommand{\calp}{\mathcal{P}}

\makeatletter
\newcommand{\colim@}[2]{%
  \vtop{\m@th\ialign{##\cr
    \hfil$#1\operator@font colim$\hfil\cr
    \noalign{\nointerlineskip\kern1.5\ex@}#2\cr
    \noalign{\nointerlineskip\kern-\ex@}\cr}}%
}
\newcommand{\colim}{  \mathop{\mathpalette\colim@{\rightarrowfill@\textstyle}}\nmlimits@}
\makeatother

\newcounter{commentcounter}

\theoremstyle{plain}
\newtheorem{theorem}{Theorem}[section]

\newtheorem{lemma}[theorem]{Lemma}

\newtheorem{corollary}[theorem]{Corollary}
\newtheorem{proposition}[theorem]{Proposition}

\newtheorem*{theorem*}{Theorem}
\newtheorem*{mtheorem*}{Main Theorem}

\theoremstyle{definition}
\newtheorem{definition}[theorem]{Definition}

\newtheorem{example}[theorem]{Example}

\newtheorem{remark}[theorem]{Remark}

\theoremstyle{remark}
\newtheorem*{summary*}{Summary}
\newtheorem{innercustomgeneric}{\customgenericname}
\providecommand{\customgenericname}{}
\newcommand{\newcustomtheorem}[2]{%
  \newenvironment{#1}[1]
  {%
   \renewcommand\customgenericname{#2}%
   \renewcommand\theinnercustomgeneric{##1}%
   \innercustomgeneric
  }
  {\endinnercustomgeneric}
}

\newcustomtheorem{customthm}{Theorem}
\newcustomtheorem{customcor}{Corollary}

\makeatletter\let\c@equation=\c@theorem\makeatother

\newcommand{\version}[1] 
{\begin{center} Last edited on #1\\
    Last compiled on \today\\
file name: \jobname
  \end{center}
}

\begin{document}

\begin{abstract}  
We study the fixed points of  the universal $G$-equivariant
 complex vector bundle of rank $n$ and obtain a decomposition
formula in terms of twisted equivariant universal complex vector bundles of smaller rank. 
We use this decomposition to describe the fixed points of the complex equivariant 
K-theory spectrum and the equivariant unitary bordism groups for adjacent families of 
subgroups.

 \end{abstract}

\maketitle

\section*{Introduction}

In this article decomposition formulas for equivariant $K$-theory and 
geometric equivariant bordism of stably almost complex manifolds
are obtained under suitable hypotheses.
The underlying main technical idea behind such decompositions is a splitting 
formula for equivariant complex vector bundles first obtained in \cite{GomezUribe} 
for the particular case of finite groups. In this article we generalize this splitting 
formula for the general case of compact Lie groups and apply it to obtain the 
decompositions of equivariant $K$-theory and equivariant unitary bordism mentioned above. 

More precisely, suppose that $G$ is a compact Lie group that fits in a short exact sequence 
of compact Lie groups $1 \to A\stackrel{\iota}\rightarrow G\stackrel{\pi}\rightarrow Q \to  1$.
Let $X$ be a compact  $G$-space such that $A$ acts trivially on $X$. 
In the first part of this article we study $G$-equivariant complex vector bundles $p:E\to X$.  
Since $A$ acts trivially on $X$ the fibers of $E$ can be seen as $A$-representations. 
By decomposing $E$ into $A$-isotypical pieces we obtain a splitting of $E$ as an 
$A$-equivariant vector bundle in the form 
$\bigoplus_{[\tau]\in \Irr(A)}\IV_{\tau}\otimes \Hom_{A}(\IV_{\tau},E)\cong E$. 
Here $\IV_{\tau}$ denotes the trivial $A$-vector bundle $\pi_{1}:X\times V_{\tau}\to X$ 
associated to an irreducible representation $\tau:A\to U(V_{\tau})$ and $\Irr(A)$ denotes the 
set of isomorphism classes of complex irreducible $A$-representations. This 
splitting is one of $A$-vector bundles and not one of
$G$-vector bundles since in general the bundles
$\IV_{\tau}\otimes \Hom_{A}(\IV_{\tau},E)$ do not possess the structure of a $G$-vector bundle,
(see example \ref{exampletwopoints}). A key technical observation of this work is that, 
up to isomorphism, the direct sum 
$\bigoplus_{[\tau]\in \Irr(A)}\IV_{\tau}\otimes \Hom_{A}(\IV_{\tau},E)$ can be rearranged 
using the different orbits of the action of $Q$ on $\Irr(A)$ as to  obtain a decomposition of 
$E$ in terms of $G$-vector bundles. This way a splitting of $E$ as a $G$-equivariant vector 
bundle is obtained in Theorem \ref{theorem decomposition vector bundle}. 
This result plays a key role in this paper. 

Given an irreducible representation $\rho:A\to U(V_{\rho})$ we can obtain in a natural way a 
central extension of the form $1\to \IS^{1}\to \widetilde{Q}_{\rho}\to Q_{\rho}\to 1$, 
where $Q_{\rho}=G_{\rho}/A$ and $G_{\rho}=\{g\in G ~|~g\cdot \rho\cong \rho\}$, (see Sections 
\ref{Preliminaries} and \ref{twsitedKtheory} for definitions). It turns out that each of the 
pieces in the splitting formula given in Theorem \ref{theorem decomposition vector bundle} 
can be used to define a twisted form of an equivariant $K$-theory, and as a consequence the 
following result in this article is obtained:

\begin{customcor}{2.8}
Let $G$ be a compact Lie group, $X$ a $G$ space on which the normal 
subgroup $A$ acts trivially. Then there is a natural isomorphism
$$ K^*_G(X) \cong \bigoplus_{\rho \in G \backslash \Irr(A)}
{}^{\widetilde{Q}_{\rho}} K^*_{Q_\rho}(X)$$
where  $\rho$ runs over representatives of the orbits of the 
$G$-action on the set of isomorphism classes of irreducible $A$-representations 
and $Q_\rho=G_\rho/A$. 
\end{customcor}

In the above formula  
${}^{\widetilde{Q}_{\rho}} K^*_{Q_\rho}(X)$ denotes a twisted form of 
${Q_\rho}$-equivariant $K$-theory. This result generalizes a similar decomposition 
obtained in  \cite{GomezUribe} for the particular case of finite groups. 

On the other hand, the decomposition obtained in 
Theorem \ref{theorem decomposition vector bundle} can be carried out at the level 
of the universal $G$-equivariant complex bundle of rank $n$ denoted by 
$\gamma_GU(n) \to B_GU(n)$. Here $B_GU(n)$ is the classifying space of $G$-equivariant 
rank $n$ complex vector bundles. Applying this decomposition to the restriction of 
$\gamma_GU(n)$ to $B_GU(n)^A$, we obtain a $N_A/A$-equivariant homotopy equivalence 
with a product of classifying spaces parametrized by the orbits of the action of the 
normalizer $N_A$ on the set of non-trivial irreducible representations of $A$. 
This result is also one of the main results of this article and is summarized 
in Theorem \ref{theorem decomposition of B_GU(n)}.

The second part of this article adds to the understanding of the 
geometric equivariant bordism groups of stably almost complex manifolds with boundary, 
whenever the isotropy groups of the interior of the manifold differ by one conjugacy 
class of subgroups to the isotropies of the boundary. The equivariant version of the 
bordism theories was developed by Conner and Floyd in 
their monumental work \cite{ConnerFloyd-book, ConnerFloyd-Odd} and the unitary 
equivariant bordism theory was developed by Stong \cite{Stong-complex} among others. 
A compact $G$-equivariant manifold is unitary if the tangent bundle may be stabilized 
with trivial real bundles thus becoming isomorphic to a $G$-equivariant complex
vector bundle. The bordism group of unitary $G$-equivariant manifolds is denoted 
$\Omega^G_*$ and the product of manifolds makes  $\Omega^G_*$ into a ring and moreover 
a $\Omega_*$-module. The calculation of the  $\Omega_*$-module structure of $\Omega^G_*$ 
has been evasive and very little is known whenever $G$ is not abelian.
Whenever $G$ is abelian it is known that $\Omega^G_*$ is zero in odd degrees and a 
free $\Omega_*$-module in even degrees, 
(see \cite[ \S XXVIII, Thm. 5.3]{May-book}, \cite[Thm. 1]{Ossa}), and the question
remains open whether this is also the case whenever $G$ is not abelian.

The main calculational tool to understand $\Omega^G_*$ is to restrict the attention to unitary
manifolds on which the isotropy groups at each point lie on a prescribed family of subgroups 
of $G$. For a pair of families $(\calf, \calf')$ of subgroups of $G$ with $\calf' \subset \calf$ 
denote by $\Omega_*^G\{\calf, \calf'\}$
the bordism classes of unitary $G$-manifolds with boundary $(M,\partial M)$ such that the 
isotropy groups of the points in
$M$ lie on $\calf$ and the isotropy groups of the points of the boundary $\partial M$ lie 
on $\calf'$. Whenever
the families differ by the set of groups conjugate to a fix group $A$ they are called adjacent. 
Whenever $A$ is normal in $G$
and $(\calf, \calf')$ the pair of families is adjacent differing by $A$, the bordism class of a
manifold $(M,\partial M)$ in $\Omega_*^G\{\calf, \calf'\}$ is equivalent to the
bordism class of the disk bundle of the tubular neighborhood of the fixed point set $M^A$ in $M$. 
Therefore we may keep the information of the normal bundle by a map from $M^A$ to the 
classifying space of $G$-equivariant complex vector bundles over trivial $A$-spaces.  Hence the 
unitary $G$-equivariant bordism groups for adjacent families can be written in terms of 
non-equivariant unitary bordism groups of a product of certain classifying spaces. As a 
consequence of Theorem \ref{theorem decomposition vector bundle} the following decomposition of 
$G$-equivariant bordisms is obtained. This theorem is the last main result in this article 
and is a new result for compact Lie groups that are not abelian.

\begin{customthm} {4.6}
Suppose that $G$ is a compact Lie group and let $A$ be a closed normal subgroup of $G$.
If $(\calf,\calf')$ is an adjacent pair of families of subgroups  of 
$G$ differing by $A$, then
$$\Omega_n^G\{\calf,\calf' \}(X) \cong \bigoplus_{0 \leq 2k \leq n -{\rm{dim}}(G/A) }
\Omega_{n -2k}^{G/A}\{\{1\}\}\left(X^A \times \bigsqcup_{\overline{P} 
\in \overline{\calp}(k,A)}B_{G/A}U(\overline{P})\right),$$
where $\{1\}$ is the family of subgroups of $G/A$ which only contains the trivial group. 
\end{customthm}

In the above theorem $B_{G/A}U(\overline{P})$ denotes a product of classifying spaces 
for different twisted equivariant vector bundles, (see Section  \ref{section bordism}  for the 
precise definition).

In the last section we use the previous theorem to determine the $\Omega_*$-module structure of 
$\Omega^{D_{2p}}_*$, where $D_{2p}$ is the dihedral group of order $2p$ with $p$ an odd prime.
We show that $\Omega^{D_{2p}}_*$ is a free $\Omega_*$-module in even degrees and zero in odd 
degrees. 

This paper is organized as follows: in Section \ref{Preliminaries} we review some preliminaries 
related to central extensions and twisted equivariant $K$-theory. In Section \ref{twsitedKtheory} 
we prove Theorem \ref{theorem decomposition vector bundle} and obtain 
Corollary \ref{corollary decomposition K-theory} as a consequence. In Section 
\ref{section classifying spaces} we calculate the homotopy type of the fixed 
points space $B_GU(n)^A$ for a closed subgroup $A$ of $G$. Section \ref{section bordism} 
is dedicated to study geometric $G$-equivariant bordism and Theorem 
\ref{thm bordism for adjacent families} is proved there. Finally, 
in Section \ref{applications} some applications are considered.

\section{Preliminaries}\label{Preliminaries}
  
\subsection{Central extensions and representations}

Suppose that we have an exact sequence of compact Lie groups
$$
1 \to A\stackrel{\iota}\rightarrow G\stackrel{\pi}\rightarrow Q \to  1
$$
and let $\rho : A \to U(V_{\rho})$ be a complex, finite dimensional, 
irreducible representation of $A$. Since $A$ is normal in $G$, the group $G$ 
acts on the left on the set $\Hom(A,U(V_{\rho}))$ of homomorphisms from $A$ to $U(V_{\rho})$ 
by the equation $$ (g \cdot \chi)(a) := \chi(g^{-1}ag) $$ for $\chi \in \Hom(A,U(V_{\rho}))$ 
and $a \in A$. Also, the unitary group $U(V_{\rho})$ acts on the right on 
$\Hom(A,U(V_{\rho}))$ by conjugation by the equation $$ (\chi \cdot M)(a) := M^{-1}\chi(a)M $$
for $\chi \in \Hom(A,U(V_{\rho}))$ and $ M \in U(V_{\rho})$.
Note further that this left $G$ action on  $\Hom(A,U(V_{\rho}))$ commutes with 
the right $U(V_{\rho})$ action.

In this section we are going to show that if the representation $\rho$ 
is such that $g\cdot \rho\cong \rho$ for every $g\in G$, then we can associate to 
$\rho$ a central extension of $G$ by $\IS^{1}$ and that this central extension can 
be thought as an obstruction for the existence  of an extension $\tilde{\rho}:G\to U(V_{\rho})$ 
of $\rho$. For this notice that the projective unitary group 
$PU(V_{\rho}):=U(V_{\rho})/Z(U(V_{\rho}))=U(V_{\rho})/\IS^{1}$  
can be identified with the inner automorphisms of $U(V_{\rho})$ via 
the map $p(M)=\Ad_M$, where $\Ad_M(N)=MNM^{-1}$ for $M\in U(V_{\rho})$.

\begin{lemma} \label{lemma from G to Inn(U)}
Suppose that for all $g \in G$ the irreducible representation $g \cdot \rho$
is isomorphic to $\rho$. Then there is a unique homomorphism
$f: G \to PU(V_{\rho})$  making the following diagram commutative
$$\xymatrix{ A \ar[d]_\rho \ar[r]^{\iota} & G \ar[d]^f \\
U(V_\rho) \ar[r]^p & PU(V_{\rho}).
 }$$ 
\end{lemma}
\begin{proof}
Suppose that $g\in G$. Note that the representation $g \cdot \rho : A \to U(V_\rho)$, 
defined by $(g\cdot \rho)(a)=\rho(g^{-1}ag)$ for $a\in A$ and $g\in G$,
also has $V_\rho$ for underlying vector space. By 
Schur's lemma we know that $\Hom_{U(V_\rho)}(g \cdot \rho, \rho) \cong \IC$,
thus there is only one inner automorphism $f(g^{-1})\in \Inn(U(V_\rho))$ 
of $U(V_\rho)$ such that $g \cdot \rho = f(g^{-1}) \circ \rho$.
Whenever $g \in A$, we have that $ g \cdot \rho = \Ad_{\rho(g)^{-1}} \circ \rho$ and
therefore we set $f(g)=\Ad_{\rho(g)}$ whenever $g \in A$.

For $h,g \in G$ we know that $(hg \cdot \rho)=h \cdot ( g \cdot \rho)$ thus implying that
$f((hg)^{-1}) \circ \rho= h \cdot (f(g^{-1}) \circ \rho)=  f(g^{-1})\circ  f(h^{-1}) \circ \rho$
and therefore $f((hg)^{-1}) =  f(g^{-1})\circ  f(h^{-1})$.
Hence $f$ is a homomorphism and by definition it is unique. 
\end{proof}

Suppose now that we have an irreducible representation  $\rho:A \to U(V_{\rho})$ such that 
$g\cdot \rho\cong \rho$ for every $g\in G$. Let $f: G \to PU(V_{\rho})$ be the homomorphism 
constructed in the previous lemma so that the following diagram commutes
\[
\xymatrix{ A \ar[d]_\rho \ar[r]^{\iota} & G \ar[d]^f \\
U(V_\rho) \ar[r]^p & PU(V_{\rho}).
 }
\] 
Recall that the natural projection map 
\[
1\to \IS^{1}\to U(V_{\rho})\stackrel{p}{\rightarrow} PU(V_{\rho})\to 1
\]
defines a central extension of $PU(V_{\rho})$ by $\IS^{1}$. 
Define the Lie group $\widetilde{G}_{\rho} : = f^*U(V_\rho)$ as the pullback of $U(V_\rho)$ 
under the homomorphism $f$ so that we obtain 
a central extension of Lie groups
\[
1\to \IS^{1}\to \widetilde{G}_{\rho} \stackrel{\tau_{\rho}}{\rightarrow} G\to 1. 
\]
If we denote by $\widetilde{f} : \widetilde{G}_{\rho} \to U(V_\rho)$ the
induced homomorphism we obtain the
following commutative diagram in the category of Lie groups
\begin{align} \label{diagram f tilde}
\xymatrix{
& \IS^1  \ar[d]  &  \IS^1 \ar[d] \\ 
A \ar[r]^{\widetilde{\iota}}  \ar[d]^= & \widetilde{G}_{\rho}
\ar[r]^{\widetilde{f}} \ar[d] & U(V_\rho) \ar[d] \\
A \ar[r]^\iota  & G \ar[r]^f & PU(V_\rho).
}
\end{align}
In the above diagram the vertical sequences are 
$\IS^1$-central extensions and the homomorphism 
$\widetilde{\iota} : A \to  \widetilde{G}_{\rho}$ is the unique homomorphism 
such that $ \rho = \widetilde{f} \circ \widetilde{\iota}$.
Since $A$ is normal in $G$ and $\widetilde{G}_{\rho}$ is a central extension of $G$, 
then we have that $\widetilde{\iota}(A)$ is also normal in $\widetilde{G}_{\rho}$. 
Therefore the quotient $\widetilde{G}_{\rho}/\widetilde{\iota}(A)$
is a Lie group and we denote it by $$\widetilde{Q}_{\rho}:= 
\widetilde{G}_{\rho}/\widetilde{\iota}(A)$$ since it depends only on $\rho$, 
and it fits into the diagram
\begin{equation}\label{diagram2}
\xymatrix{
& \IS^1  \ar[d]  &  \IS^1 \ar[d] \\ 
A \ar[r]^{\widetilde{\iota}}  \ar[d]^= & \widetilde{G}_{\rho} 
\ar[r]^{\widetilde{\pi}} \ar[d] & \widetilde{Q}_{\rho}\ar[d] \\
A \ar[r]^\iota  & G \ar[r]^\pi & Q
}
\end{equation}
where the horizontal sequences are exact, the vertical are $\IS^1$-central extensions
and the square on the right hand side is a pullback square.

\begin{proposition}\label{obstruction for representations}
Consider the short exact sequence $1 \to A \to G \to Q \to 1$ of compact Lie 
groups and $\rho:A \to U(V_\rho)$ an irreducible representation of $A$ such that 
its isomorphism class is invariant under the $G$ action, namely that 
$(g\cdot \rho) \cong \rho$ for all $g \in G$.
Then the representation $\rho$ may be extended to an irreducible representation
$\widetilde{\rho} : G \to U(V_\rho)$ if and only if the $\IS^1$-central extension 
$\widetilde{Q}_{\rho}$ is trivial, i.e. $\widetilde{Q}_{\rho}$ is isomorphic to 
$Q \times \IS^1$ as Lie groups. 
\end{proposition}
\begin{proof}
If $\widetilde{Q}_{\rho}$ is trivial as a $\IS^1$-central extension, then 
$\widetilde{G}_{\rho}$ must also be trivial as a $\IS^1$-central extension, i.e.
$\widetilde{G}_\rho\cong G \times \IS^1$. Therefore there is a homomorphism 
$\sigma:G \to \widetilde{G}_{\rho}$ compatible with the quotient homomorphism 
$\widetilde{G}_{\rho} \to G$ whose composition 
$\widetilde{\rho}:=\widetilde{f} \circ \sigma : G \to U(V_\rho)$ is the desired 
extension of $\rho$.

Conversely, if  $\widetilde{\rho} : G \to U(V_\rho)$ extends the homomorphism $\rho$
then $\widetilde{\rho} $ defines a homomorphism $\sigma: G \to \widetilde{G}_\rho$ 
compatible with the quotient homomorphism $\widetilde{G}_\rho \to G$ thus 
making $\widetilde{G}$ a trivial $\IS^1$-central extension. It follows that 
$\widetilde{Q}$ is also trivial as a $\IS^1$-central extension. 
\end{proof}

\begin{remark}
Recall that isomorphism classes of $\IS^1$-central extensions of $Q$ are in 1-1 
correspondence with elements in $H^3(BQ, \IZ)$ (see 
\cite[Prop. 6.3]{Atiyah-Segal2}). By the previous proposition we may 
say that the obstruction for the existence of the extension 
$\widetilde{\rho} : G \to U(V_\rho)$ of the irreducible representation 
$\rho:A \to U(V_\rho)$ is the cohomology class 
$[\widetilde{Q}_\rho] \in H^3(BQ,\IZ)$ which encodes the information 
of the $\IS^1$-central extension $\IS^1 \to \widetilde{Q}_\rho \to Q$. 
In \cite[\S2]{GomezUribe} the obstruction of the existence of extensions 
of representations $\rho:A\to U(V_{\rho})$ was studied for the 
 case of finite groups and the obstruction was explicitly described in terms of cocycles.
\end{remark}

\subsection{Twisted equivariant K-theory}\label{twistedK}

Next we recall the definition of twisted equivariant $K$-theory that we 
will use throughout this article. For this suppose that $Q$ is a 
compact Lie group and let \begin{equation*}
1\to \IS^{1}\to \widetilde{Q}\stackrel{\tau}{\rightarrow} Q\to 1
\end{equation*}
be a $\IS^1$-central extension of $Q$.
Let $X$ be a $Q$-space and endow it with the action of $\widetilde{Q}$ 
induced by the $Q$-action. Consider 
the set of isomorphism classes of $\widetilde{Q}$-vector bundles $p:E\to X$ on 
which the elements $z\in \IS^{1}$ act by the scalar multiplication of $z^{-1}$, 
denote this set by ${}^{\widetilde{Q}}\text{Vec}_{Q}(X)$. 
The set ${}^{\widetilde{Q}}\text{Vec}_{Q}(X)$ is a semigroup under direct sum of 
vector bundles and we define the twisted equivariant K-group 
${}^{\widetilde{Q}}K_{Q}^{0}(X)$ as the Grothendieck construction applied 
to  ${}^{\widetilde{Q}}\text{Vec}_{Q}(X)$.  For $n>0$ the twisted groups 
${}^{\widetilde{Q}}K_{Q}^{n}(X)$ are defined as 
${}^{\widetilde{Q}}\widetilde{K}_{Q}^{0}(\Sigma^{n}X_{+})$. We call the groups 
${}^{\widetilde{Q}}K_{Q}^{*}(X)$ the  $\widetilde{Q}$-twisted
$Q$-equivariant K-theory groups of $X$. Notice that ${}^{\widetilde{Q}}K_{Q}^{*}(X)$  
is naturally a module over $R(Q)$. The twisted groups ${}^{\widetilde{Q}}K_{Q}^{0}(X)$ 
can alternatively be defined as follows. The action of $\IS^{1}$ on $X$ obtained by 
restricting the $\widetilde{Q}$-action   is trivial. Therefore we obtain a natural map 
$K_{\widetilde{Q}}^{0}(X)\to K_{\IS^{1}}^{0}(X)\cong K^{0}(X)\otimes R(\IS^{1})$. 
Composing this with the restriction 
map $K^{0}(X)\otimes R(\IS^{1})\to R(\IS^{1})$ we obtain a natural map 
$K_{\widetilde{Q}}^{0}(X)\to R(\IS^{1})$ and ${}^{\widetilde{Q}}K_{Q}^{0}(X)$  can also 
be defined as the inverse image of the subgroup generated by the $\IS^{1}$-representations 
on which a scalar $z$ acts by multiplication of $z^{-1}$. 
This description can also be used to define ${}^{\widetilde{Q}}K_{Q}^{n}(X)$. The cohomology 
class that classifies the twist is the image of the class 
$[\widetilde{Q}] \in H^3(BQ,\IZ)$, that corresponds to the central extension 
$\widetilde{Q}$, under the canonical map $H_{Q}^3(*,\IZ)=H^3(BQ,\IZ)\to H_{Q}^3(X,\IZ)$.

\begin{remark}
In the literature it is more common to encounter a different (but equivalent) definition of this 
twisted form of equivariant K-theory, cf. \cite[Def. 7.1]{Adem-Ruan}. Suppose that we are  
given a central extension $1\to \IS^{1}\to \widetilde{Q}\stackrel{\tau}{\rightarrow} Q\to 1$. 
We can also consider the set ${}^{\widetilde{Q}^{+}}\text{Vec}_{Q}(X)$ of isomorphism classes 
of  $\widetilde{Q}$-vector bundles over $p:E\to X$ on which an element 
$z\in \IS^{1}$ acts by the scalar multiplication of $z$. The set 
${}^{\widetilde{Q}^{+}}\text{Vec}_{Q}(X)$ is also a semigroup under direct sum of 
vector bundles and we can also define a twisted form of K-theory, which we will denote  
by ${}^{\widetilde{Q}^{+}}K_{Q}^{0}(X)$, as 
the Grothendieck construction applied to  ${}^{\widetilde{Q}^{+}}\text{Vec}_{Q}(X)$.
For $n>0$ the twisted groups ${}^{\widetilde{Q}^{+}}K_{Q}^{n}(X)$ can be defined 
in a similar way as above. These two twisted forms of equivariant K-theory 
are naturally isomorphic as we show next. By definition it 
suffices to prove the case $n=0$. Let $p:E\to X$ be a 
$\widetilde{Q}$-vector bundle such that a central element 
$z\in \IS^{1}$  acts by the scalar multiplication of $z^{-1}$ 
so that $[E]\in {}^{\widetilde{Q}}\text{Vec}_{Q}(X)$. 
Let $\Hom(E,\underline{\IC})$ be the vector bundle dual to $E$, where  
$\underline{\IC}$ denotes the trivial $\widetilde{Q}$-vector bundle 
$\pi_{1}:X\times \IC \to X$.  If $\phi\in \Hom(E,\underline{\IC})_{x}$ and 
$q\in \widetilde{Q}$ the action of $q$ on $\phi$ is the element 
$q\cdot \phi\in \Hom(E,\underline{\IC})_{qx}$ defined by 
$(q\cdot \phi)(v)=\phi(q^{-1}\cdot v)$ for every $v\in E_{qx}$. 
With this action $\Hom(E,\underline{\IC})$ is  a $\widetilde{Q}$-vector bundle. 
If $z\in \IS^{1}$ is a central element and 
$\phi\in \Hom(E,\underline{\IC})_{x}$, then as the action of $z$ in $E$ is 
given by scalar multiplication of $z^{-1}$, we have 
\[
(z\cdot \phi)(v)=\phi(z^{-1}\cdot v)=\phi(zv)=z\phi(v)
\]
for every $v\in E_{x}$. This shows that $\Hom(E,\underline{\IC})$ is a 
$\widetilde{Q}$-equivariant vector bundle on which the central factor $\IS^{1}$ 
acts by multiplication of scalars on the fibers so that 
$[\Hom(E,\underline{\IC})]\in {}^{\widetilde{Q}^{+}}\text{Vec}_{Q}(X)$. 
The assignment 
\begin{align*}
{}^{\widetilde{Q}}\text{Vec}_{Q}(X)&\stackrel{\cong}{\to} 
{}^{\widetilde{Q}^{+}}\text{Vec}_{Q}(X)\\
[E]&\mapsto [\Hom(E,\underline{\IC}) ]
\end{align*}
is an isomorphism of semigroups. After applying the Grothendieck construction 
we obtain an isomorphism 
${}^{\widetilde{Q}}K_{Q}^{0}(X)\stackrel{\cong}{\to} {}^{\widetilde{Q}^{+}}K_{Q}^{0}(X)$. 
Throughout this article we will work with the twisted form of equivariant 
K-theory constructed using vector bundles on which the elements of the 
central factor $\IS^{1}$ act by multiplication of their inverse; these are the 
bundles that appear naturally in our work. 
\end{remark}

\section{Equivariant K-theory with prescribed fibers}\label{twsitedKtheory}

The goal of this section is to generalized the decomposition of $G$-equivariant 
$K$-theory obtained in \cite[Theorem 3.2]{GomezUribe} to the case of compact Lie groups.

To start assume that $G$ is a compact Lie group and let $A$ be a normal 
subgroup of  $G$  so that we have an extension of compact Lie groups
\[
1\to A\stackrel{\iota}\rightarrow G\stackrel{\pi}\rightarrow Q\to  1,
\]
where $Q=G/A$. Let $G$ act on a compact space $X$ in such a way that 
$A$ acts trivially on $X$. 
Assume that $p:E\to X$ is a $G$-equivariant 
complex vector bundle.  We can give $E$ a Hermitian metric that is invariant 
under the action of $G$, in particular this metric is $A$-invariant. 
If we see $p:E\to X$ as an $A$-vector bundle then 
as the action of $A$ on $X$ is trivial, by \cite[Proposition 2.2]{SegalK} 
we have a natural isomorphism of $A$-vector bundles
\begin{align*}
\beta:\bigoplus_{[\tau]\in \Irr(A)}\IV_{\tau}\otimes \Hom_{A}(\IV_{\tau},E)& 
\stackrel{\cong}{\to} E\\
v\otimes f&\mapsto f(v).
\end{align*}
In the above equation $\Irr(A)$ denotes the set of isomorphism classes of complex 
irreducible $A$-representations, and if  $\tau:A\to U(V_{\tau})$ is an irreducible 
$A$-representation, then $\IV_{\tau}$ denotes the trivial $A$-vector bundle 
$\pi_{1}:X\times V_{\tau}\to X$. The decomposition of the vector bundle $E$ 
provided above is a decomposition as an $A$-equivariant 
bundle and not as a $G$-equivariant bundle. Furthermore, the summands 
$\IV_{\tau}\otimes \Hom_{A}(\IV_{\tau},E)$ that appear in this decomposition 
do not have in general the structure of a $G$-vector bundle in such a way that 
the map $\beta$ is $G$-equivariant. Following \cite{GomezUribe} we have the next 
definition.

\begin{definition}\label{definitiontwisted1}
Suppose that $\rho:A\to U(V_{\rho})$ is a complex irreducible representation
and that $1 \to A \to G \to Q \to 1$ is a short exact sequence of compact Lie groups.
A $(G, \rho)$-equivariant vector bundle over $X$ is a $G$-vector bundle 
$p:E\to X$ such that 
the map 
\begin{align*}
\beta:\IV_{\rho}\otimes \Hom_{A}(\IV_{\rho},E)&\to E\\
v\otimes f&\mapsto f(v)
\end{align*}
is an isomorphism of $A$-vector bundles.  
\end{definition}

A $(G,\rho)$-equivariant vector bundle is a $G$-equivariant 
vector bundle $p:E\to X$ such that for every $x\in X$ 
the $A$-representation $E_{x}$ is isomorphic to a direct sum of copies of the representation 
$\rho$. Notice that if $p:E\to X$ is a $(G,\rho)$-equivariant vector bundle then 
for every $g\in G$ we have $g\cdot \rho\cong \rho$ so that 
$(G,\rho)$-equivariant vector bundles  can only exists when this happens. 
We can define a direct summand of the equivariant K-theory using  
$(G, \rho)$-equivariant vector bundles. For this let 
$\text{Vec}_{G,\rho}(X)$ denote the set of isomorphism classes of 
$(G,\rho)$-equivariant vector bundles, where two 
$(G,\rho)$-equivariant vector bundles are isomorphic if they are isomorphic
as $G$-vector bundles. Notice that if $E_{1}$ and $E_{2}$ are two  
$(G,\rho)$-equivariant vector bundles then so is the direct sum $E_{1}\oplus E_{2}$. 
Therefore $\text{Vec}_{G, \rho}(X)$  is a semigroup. 

\begin{definition} 
Assume that $G$ acts on a compact space $X$ in such a way 
that $A$ acts trivially on $X$. 
We define $K_{G, \rho}^{0}(X)$, the $(G,\rho)$-equivariant K-theory
of $X$, as the Grothendieck construction applied to $\text{Vec}_{G,\rho}(X)$. For 
$n>0$ the group $K_{G,\rho}^{n}(X)$ is defined as 
$\widetilde{K}_{G, \rho}^{0}(\Sigma^{n}X_{+})$, where as usual 
$X_{+}$ denotes the space $X$ with an added base point.
\end{definition}
 
Following \cite{GomezUribe} we relate the $(G,\rho)$-equivariant K-theory
of $X$ with  a suitable twisted form of equivariant K-theory as defined in 
\S \ref{twistedK}. For this suppose that 
$\rho:A\to U(V_{\rho})$ is an irreducible representation
and let $f:G\to PU(V_{\rho})$ be the homomorphism associated to $\rho$  as constructed 
in Lemma \ref{lemma from G to Inn(U)}. Consider $\widetilde{G}_{\rho}=f^{*}U(V_{\rho})$ 
and $\widetilde{Q}_\rho:= \widetilde{G}_\rho/\widetilde{\iota}(A)$ so that we have 
a commutative diagram of central extensions as in diagram (\ref{diagram2}).

\begin{theorem}\label{centralaction}
Let $X$ be a $G$-space such that $A$ acts trivially on $X$.  Assume that 
$g\cdot \rho\cong \rho$ for every $g\in G$. If $p:E\to X$ is a 
$(G,\rho)$-equivariant vector bundle, then $\Hom_{A}(\IV_{\rho},E)$  has the 
structure of a $\widetilde{Q}_{\rho}$-vector bundle on which the elements 
of the central factor $\IS^{1}$ act by multiplication by their inverse.  
Moreover,  the assignment $$[E]\mapsto [\Hom_{A}(\IV_{\rho},E)]$$ defines 
a natural one to one correspondence between isomorphism classes of 
$(G, \rho)$-equivariant vector bundles over $X$ and isomorphism classes 
of $\widetilde{Q}_{\rho}$-equivariant 
vector bundles over $X$ for which the elements of the central $\IS^{1}$ 
act by multiplication of their inverse. 
\end{theorem}
\begin{proof}
Suppose that  $p:E\to X$ is a $(G,\rho)$-equivariant vector bundle. Then 
$\Hom_{A}(\IV_{\rho},E)$ is a complex vector bundle over $X$. 
Next we give  $\Hom_{A}(\IV_{\rho},E)$ an action of 
$\widetilde{G}_{\rho}$ on  which $\widetilde{\iota}(A)$ acts trivially.

Take
$\phi \in \Hom_{A}(\IV_{\rho},E)_{x}$ and $\widetilde{g}\in \widetilde{G}_\rho$, 
and define $\widetilde{g}\bullet \phi \in \Hom_{A}(\IV_{\rho},E)_{g\cdot x}$ 
by 
\[
(\widetilde{g}\bullet \phi )(v)=g \phi (\widetilde{f}(\widetilde{g})^{-1}v),
\] 
where $\widetilde{g}$ projects to $g$ in $G$ and 
$\widetilde{f} : \widetilde{G}_\rho \to U(V_\rho)$ is the 
homomorphism defined in diagram \eqref{diagram f tilde}.
The action is a composition of continuous maps and therefore it is continuous. 
It is straightforward to check that it is a homomorphism.

Now let us take $a \in A$ and consider the action of 
$\widetilde{\iota}(a)$ on $\phi$. In this case we have
$$(\widetilde{\iota}(a) \bullet \phi)(v) = a  
\phi(\widetilde{f}(\widetilde{\iota}(a))^{-1}v)= a\phi(\rho(a)^{-1}v)= \phi(v)$$
which implies that $\widetilde{\iota}(a) \bullet \phi = \phi$. Hence the action
of $\widetilde{\iota}(A)$ on $ \Hom_{A}(\IV_{\rho},E)$ is trivial and therefore
there is an induced action of $\widetilde{Q}_\rho = \widetilde{G}_\rho/\widetilde{\iota}(A)$
on $\Hom_{A}(\IV_{\rho},E)$ compatible with the action of $Q$ on $X$.

Now, if $\lambda \in \ker( \widetilde{G}_\rho \to G)$ the action becomes
$$(\lambda \bullet \phi)(v) = \phi(\widetilde{f}(\lambda)^{-1}v)
=  \lambda^{-1}\phi(v)$$
which implies that $\Hom_{A}(\IV_{\rho},E)$ is a $\widetilde{Q}_\rho$-equivariant
bundle where the elements of $\IS^1$ act by multiplication by their inverse.

Now let us take a $\widetilde{Q}_\rho$-equivariant
bundle $F \to X$ where the elements of $\IS^1$ act by multiplication by their inverse. 
Consider the vector bundle $\IV_{\rho}\otimes F$ and define a $\widetilde{G}_\rho$ action
in the following way: for $\widetilde{g} \in \widetilde{G}_\rho$ and $v \otimes e \in 
\IV_{\rho}\otimes F$ let the action be
$$\widetilde{g} \cdot (v \otimes e) : =\left( \widetilde{f}(\widetilde{g})v \right) \otimes   
\left(\widetilde{\pi}(\widetilde{g}) \cdot e \right)$$
where $\widetilde{\pi} : \widetilde{G}_\rho \to \widetilde{Q}_\rho$ is the homomorphism
induced by $\pi : G \to Q$. This is clearly a continuous $\widetilde{G}_\rho$ action
and for $\lambda \in \ker( \widetilde{G}_\rho \to G)$ we have that
$$\lambda \cdot (v \otimes e) : =\left( \widetilde{f}(\lambda)v \right) \otimes   
\left(\widetilde{\pi}(\lambda) \cdot e \right) = \lambda v \otimes \lambda \cdot e =
 \lambda v \otimes \lambda^{-1}e= v \otimes e$$
which implies that the action factors through $G=\widetilde{G}_\rho / \IS^1$.

Let us see now how is the action once restricted to $A$. Take $a \in A$ and 
consider the element $\widetilde{\iota}(a) \in \widetilde{G}_\rho$. The 
action of $a$ on $v \otimes e$ becomes 
$$\widetilde{\iota}(a) \cdot (v \otimes e)
=\left( \widetilde{f}(\widetilde{\iota}(a))v \right) \otimes   
\left(\widetilde{\pi}(\widetilde{\iota}(a)) \cdot e \right) 
= \rho(a)v \otimes e$$ which implies that the action of $A$ on the 
fibers of $\IV_{\rho}\otimes F$ is determined by the representation $\rho$. 
Hence $\IV_{\rho}\otimes F$ is a $(G,\rho)$-equivariant vector bundle.

If $E$ is a $(G, \rho)$-equivariant vector bundle then 
$\IV_{\rho}\otimes \Hom_{A}(\IV_{\rho},E)$ is also a 
$(G, \rho)$-equivariant vector bundle and the canonical map
\begin{align*}
\IV_{\rho}\otimes \Hom_{A}(\IV_{\rho},E) &\to E \\
v \otimes \phi & \mapsto \phi(v)
\end{align*}
is by definition an isomorphism of vector bundles. Note that the map 
is moreover $G$-equivariant; for $\widetilde{g} \in \widetilde{G}_\rho$ 
which projects to $g \in G$ we have that
$$g \cdot( v \otimes \phi) = \widetilde{g} \cdot (v \otimes \phi)
= \left( \widetilde{f}(\widetilde{g})v \right) \otimes   
\left(\widetilde{g} \bullet \phi \right),$$ whose evaluation becomes
$$(\widetilde{g} \bullet \phi )\left( \widetilde{f}(\widetilde{g})v \right) =
 g \phi \left( \widetilde{f}(\widetilde{g})^{-1} 
\widetilde{f}(\widetilde{g})v \right) = g \phi(v),$$
thus implying that the canonical evaluation map is a $G$-equivariant isomorphism.

Finally, if $F$  is a $\widetilde{Q}_\rho$-equivariant
bundle where the elements of $\IS^1$ act by multiplication by their inverse, 
we may consider the canonical isomorphism of vector bundles
\begin{align*}
 F &\to \Hom_{A}(\IV_{\rho},\IV_{\rho} \otimes F)\\
 e & \mapsto \phi_e : v  \mapsto v \otimes e.
\end{align*}
It is straightforward to check that it is moreover an isomorphism of 
$\widetilde{Q}_\rho$-equivariant vector bundles.

We conclude that the inverse map of the 
assignment $[E] \mapsto [\Hom_{A}(\IV_{\rho},E)]$
is precisely the map defined by the assignment $[F] \mapsto [\IV_{\rho}\otimes F]$.
The theorem follows.
\end{proof}

Theorem \ref{centralaction} provides the following identification of the 
$(G, \rho)$-equivariant K-groups of Definition \ref{definitiontwisted1} 
with the twisted groups ${}^{\widetilde{Q_{\rho}}}K_{Q}^{*}(X)$ defined 
in Subsection \ref{twistedK}. 
 
\begin{corollary}\label{equivalencetwistedK}
Let $G$ be a compact Lie group and $X$ be a compact $G$-space such that 
the normal subgroup $A$ of $G$ acts trivially on $X$.
Assume furthermore that $\rho : A \to U(V_\rho)$ is an irreducible representation
whose isomorphism class is fixed by $G$, i.e.
$g\cdot \rho\cong \rho$ for every $g\in G$. Then the homomorphism
\begin{align*}
K_{G, \rho}^{*}(X)&\stackrel{\cong}{\to}  {}^{\widetilde{Q}_{\rho}}K_{Q}^{*}(X) \\
[E]&\mapsto [\Hom_{A}(\IV_{\rho},E)] 
\end{align*}
is a natural isomorphism of $R(Q)$-modules. The inverse map is $F \mapsto \IV_{\rho} \otimes F$.
\end{corollary}

Whenever the isomorphism class of the $A$-representation $\rho$ is not fixed 
by the whole group $G$ we need to be more careful. Define 
$$G_{\rho}:= \{ g \in G | g \cdot \rho \cong \rho \} \ \ \mbox{and} \ \ Q_{\rho}:= G_\rho / A$$
and call $\widetilde{G}_\rho$ and $\widetilde{Q}_\rho$ the $\IS^1$-central extensions
which measure the obstruction for the extension of $\rho$ to $G_\rho$ 
constructed in Proposition \ref{obstruction for representations}.

Now consider a $G$-vector bundle $E$ over the compact $G$-space $X$ on which $A$ acts trivially.
We know that as a $A$-vector bundle we have the isomorphism
\begin{equation}\label{A-decomposition}
\bigoplus_{\rho \in \Irr(A)}\IV_{\rho} \otimes\Hom_{A}(\IV_{\rho},E) \stackrel{\cong}{\to} E
\end{equation}
given by the evaluation, where $\rho$ runs over the set of 
isomorphism classes of irreducible $A$-representations.
Each of the vector bundles $\IV_{\rho} \otimes\Hom_{A}(\IV_{\rho},E)$ is a $G_\rho$-vector
bundle. However, it is not possible in general to provide each factor 
$\IV_{\rho} \otimes\Hom_{A}(\IV_{\rho},E)$ with the 
structure of a $G$-vector bundle in such a way that the isomorphism 
(\ref{A-decomposition}) is an isomorphism of $G$-equivariant vector bundles as the $G$-action 
intertwines the vector bundles associated to irreducible $A$-representations 
which are related by the action of $G$. To illustrate this issue we 
explore the following example. 

\begin{example}\label{exampletwopoints}
Suppose that $G=D_8$ is the dihedral group generated by the elements $a,b$ 
with relations $a^4=b^2=1$ and $bab=a^{-1}$ 
and let $A=\langle a \rangle=\IZ/4$ so that we have a short exact sequence 
\[
1 \to A\stackrel{\iota}\rightarrow G\stackrel{\pi}\rightarrow Q \to  1
\]
with $Q=G/A=\{1,\tau\}= \IZ/2$, where $\tau=[b]\in Q$.  Let $\rho:A\to \IS^{1}$ 
be the representation given by $\rho(a)=e^{\frac{2\pi i}{4}}=i$. In this example 
we have $\Irr(\IZ/4)= \{\mathbf{1}, \rho , \rho^2, \rho^3\}$ 
and the action of $Q$ on $\Irr(\IZ/4)$ is such that $\tau$ permutes the 
isomorphism classes of the representations $\rho$ and $\rho^{3}$. Denote 
by $V_{\rho}:=\IC$ equipped with the $A$-representation $\rho$ and 
$V_{\rho^{3}}:=\IC$ equipped with the $A$-representation $\rho^{3}$. 
Consider the balanced product $E:=D_8\times_{A} V_{\rho}$  seen as a 
$D_8$-equivariant vector bundle over $D_{8}/A=\{*,*'\}.$ Note that 
as a $D_8$-equivariant vector bundle $E$ is isomorphic to the bundle $V_{\rho} \sqcup V_{\rho^{3}}$
where the action of $b$ maps $V_{\rho}$ to $V_{\rho^{3}}$ and vice versa using the explicit
isomorphisms with $\IC$. Here we see $ V_{\rho}$ as a bundle over $\{*\}$ and $V_{\rho^{3}}$ 
as a bundle over $\{ *'\}$. In this case we have 
\begin{align*}
\IV_{\rho} \otimes\Hom_{A}(\IV_{\rho},E)&\cong  
V_{\rho}\sqcup \{*'\},\\
\IV_{\rho^{3}} \otimes\Hom_{A}(\IV_{\rho^{3}},E)&\cong  
\{*\} \sqcup V_{\rho^{3}}.
\end{align*}
Hence 
$$\IV_{\rho} \otimes\Hom_{A}(\IV_{\rho},E) \oplus \IV_{\rho^{3}} \otimes\Hom_{A}(\IV_{\rho^{3}},E) 
\cong V_{\rho}\sqcup \{*'\} \oplus \{*\} \sqcup V_{\rho^{3}} \cong E$$ 
as $D_{8}$-equivariant vector bundles. However, the factors $
\IV_{\rho} \otimes\Hom_{A}(\IV_{\rho},E)$ 
and $\IV_{\rho^{3}} \otimes\Hom_{A}(\IV_{\rho^{3}},E)$ 
do not possess a structure of a $D_{8}$-vector bundle that is compatible with the above isomorphism.
\end{example}

Suppose now that $p:E\to X$ is a $G$-vector bundle over the compact 
$G$-space $X$ on which $A$ acts trivially. As our next step we show that the 
factors in the decomposition described in formula (\ref{A-decomposition}) can be arranged in a 
suitable way as to obtain a decomposition of $E$ as a $G$-vector bundle. 
Choosing representatives $\{g_i\}_i$ for each class in $G/G_\rho$,  
we know that the image of the evaluation map
$$\bigoplus_{i}\IV_{g_i \cdot \rho} \otimes\Hom_{A}(\IV_{g_i \cdot \rho},E) {\to} E$$
becomes a $G$-equivariant vector bundle. Notice that $G/G_\rho$  
is finite since $G_\rho$ contains the connected component
of the identity of $G$ and $G$ is compact; this follows 
from the fact that $G$ is acting on the discrete
set $\Irr(A)$. Now,  in order to define the bundle above on 
a coordinate free fashion, we need to 
promote the $G_\rho$-equivariant bundle 
$\IV_{\rho} \otimes\Hom_{A}(\IV_{\rho},E)$ over $X$ to a 
$G$-equivariant bundle over the same space $X$. This construction was 
denoted {\it multiplicative induction} in \cite[\S 4]{Bix-tomDieck} and here we will 
recall its properties.

Let $G$ be a compact Lie group and $H$ a closed subgroup. The 
restriction functor $r_H^G$ from $G$-spaces to $H$-spaces which 
restricts the action to $H$ has a left adjoint which maps
a $H$-space $Y$ to the $G$-space $G \times_H Y$ thus having a homeomorphism
$$\map(G\times_H Y,X)^G \cong \map(Y,r_H^GX)^H  $$
for any $G$-space $X$.  This left adjoint is additive, but in general 
it is not multiplicative.
A right adjoint for the restriction functor $r_H^G$ can be defined on 
an $H$-space $Y$ as the $G$-space
of $H$-equivariant maps from $G$ to $Y$
$$m_H^G(Y) := \map(G,Y)^H$$
where $G$ is considered as an $H$-space via left multiplication. The $G$-action on $m_H^G(Y)$ 
is given by $(g \cdot f)(k):= f(kg)$. A map of $H$-spaces $\phi: Y_1 \to Y_2$ induces a map
of $G$-spaces $m_H^G(\phi) : m_H^G(Y_1) \to m_H^G(Y_2)$ by composition. In this case
there is a homeomorphism:
$$ \map(r_H^G(X),Y)^H \cong \map(X,m_H^G(Y))^G 
$$ 
for any $G$-space $X$ and any $H$-space $Y$. The maps are defined by:

$$\map(X,m_H^G(Y))^G \rightarrow \map(r_H^G(X),Y)^H, \ \ F 
\mapsto \left( x \mapsto F(x)(1_G) \right)$$
 
$$ \map(r_H^G(X),Y)^H \rightarrow \map(X,m_H^G(Y))^G , \ \ 
f \mapsto  m_H^G(f) \circ p_H^G,$$ 
where $p_H^G : X \to m_H^G(r_H^G(X))$
is defined by the equation $$(p_H^G(x))(g)=gx$$
and is the unit of the adjunction.

The functor $m_H^G$ is called multiplicative because 
$$m_H^G(Y_1 \times Y_2) \cong m_H^G(Y_1) \times m_H^G(Y_2)$$
for any $H$-spaces $Y_1$ and $Y_2$.

Note also that the space $m_H^G(Y)$ is homeomorphic to the space 
$\Gamma(G \times_H Y, G/H)$ of sections of the projection  map $G \times_H Y \to G/H$, 
endowed with the $G$-action given by $(g \cdot \sigma) (kH) = g\sigma(g^{-1}kH)$ 
where $\sigma$ is any section. In the case that $G/H$ is finite the space $m_H^G(Y)$ is 
homeomorphic to the product of $|G \colon H|$ copies of $Y$.

\medskip

Let us now consider the $G$-equivariant bundle
$$m_{G_\rho}^G(\IV_{\rho} \otimes\Hom_{A}(\IV_{\rho},E)) \to m_{G_\rho}^G(X)$$
and construct the pullback bundle
$$(p_{G_\rho}^G)^* \left(m_{G_\rho}^G(\IV_{\rho} \otimes\Hom_{A}(\IV_{\rho},E)) \right) \to X.$$

The  $G$-equivariant bundle $(p_{G_\rho}^G)^* \left(m_{G_\rho}^G(\IV_{\rho} 
\otimes\Hom_{A}(\IV_{\rho},E)) \right)$ 
is endowed with a natural $G$-equivariant map to $E$, defined
by the restriction of the natural map 
$$\map(G,\IV_{\rho} \otimes\Hom_{A}(\IV_{\rho},E))^{G_\rho} \to E; 
\ \ \phi \mapsto ev(\phi(1_G))$$
induced by the evaluation
map $ev: \IV_{\rho} \otimes\Hom_{A}(\IV_{\rho},E) \to E$, 
$ev( v \otimes f)=f(v)$. Therefore we have constructed the desired 
$G$-vector bundle over $X$.

\begin{theorem} \label{theorem decomposition vector bundle}
Let $G$ be a compact Lie group and $E$ a $G$-equivariant complex vector bundle 
over the compact $G$-space $X$.  If the action on $X$ by the normal subgroup $A$ of $G$ is 
trivial, then the following decomposition formula is an
isomorphism of $G$-equivariant bundles
$$\bigoplus_{\rho \in G \backslash \Irr(A)} (p_{G_\rho}^G)^* \left(m_{G_\rho}^G(\IV_{\rho} 
\otimes\Hom_{A}(\IV_{\rho},E)) \right)\stackrel{\cong}{\to} E$$
where $\rho$ runs over representatives of the orbits of the $G$-action on the 
set of isomorphism classes of  irreducible $A$-representations.
\end{theorem}

Note that when the Lie group $G$ is connected then $G_\rho=G$ for all $\rho$ and 
the decomposition simplifies to the isomorphism
$$\bigoplus_{\rho \in   \Irr(A)} \IV_{\rho} 
\otimes\Hom_{A}(\IV_{\rho},E)\stackrel{\cong}{\to} E$$
of $G$-equivariant bundles. This is also the case for example when $G$ is abelian.

The decomposition formula of the equivariant vector bundle 
$E$ induces an isomorphism in K-theory as follows:

\begin{corollary} \label{corollary decomposition K-theory} 
Let $G$ be a compact Lie group, $X$ a $G$ space on which the normal 
subgroup $A$ acts trivially. Then there is a natural isomorphism
$$ K^*_G(X) \cong \bigoplus_{\rho \in G \backslash \Irr(A)}
{}^{\widetilde{Q}_{\rho}} K^*_{Q_\rho}(X)$$
where  $\rho$ runs over representatives of the orbits of the 
$G$-action on the set of isomorphism classes of irreducible $A$-representations 
and $Q_\rho=G_\rho/A$. 
\end{corollary}
\begin{proof}     
The isomorphism follows from the isomorphism
\begin{align*}\bigoplus_{\rho \in G \backslash \Irr(A)}  
K^*_{G_\rho,\rho}(X)  &\stackrel{\cong}{\to} K_G^*(X)\\ 
\bigoplus_{\rho \in G \backslash \Irr(A)} E_\rho& 
\mapsto \bigoplus_{\rho \in G \backslash \Irr(A)} (p_{G_\rho}^G)^*(m_{G_\rho}^G E_\rho)
\end{align*}
and Corollary \ref{equivalencetwistedK}.
\end{proof}

\section{The decomposition at the level of classifying spaces}\label{section classifying spaces}

In this section we will write the results of the previous section at the 
level of the classifying space of $G$-equivariant complex vector bundles. 
This will show us how the spectrum of $G$-equivariant
K-theory decomposes at the fixed point set of each subgroup.

The universal bundles for twisted equivariant K-theory associated to central extensions 
of the group $G$ are constructed as follows. Consider a central extension
$$1\to \IS^1 \to \widetilde{G} \to G\to 1$$
of the compact Lie group $G$ by $\IS^{1}$. Let $\bf \widetilde{C}^\infty$ denote the 
direct sum of countable many copies of all  irreducible $\widetilde{G}$-representations
on which elements in $\IS^1= \text{ker}(\widetilde{G} \to G)$ act by scalar multiplication of 
their inverse. Let ${}^{\widetilde{G}}B_GU(n)$ denote the Grassmannian of $n$-dimensional
complex subspaces of $\bf \widetilde{C}^\infty$ and denote by 
${}^{\widetilde{G}}\gamma_GU(n)$ the universal
$n$-plane bundle over ${}^{\widetilde{G}}B_GU(n)$. The complex vector bundle 
\begin{align} \label{universal twisted equivariant U(n) bundle} \IC^{n} \to
{}^{\widetilde{G}}\gamma_GU(n) \to 
{}^{\widetilde{G}}B_GU(n)
\end{align}  is the universal $\widetilde{G}$-twisted $G$-equivariant complex vector 
bundle of rank $n$, and therefore for  a finite  $G$-CW complex $X$  we have
$${}^{\widetilde{G}}\text{Vec}_{G}^n(X) \cong [X, {}^{\widetilde{G}}B_GU(n)]_G.$$ 
Note that since $\IS^1$ acts by multiplication of the inverse of scalars, 
its action on the Grassmannian of complex $n$-planes is trivial, and therefore 
the $\widetilde{G}$-action on ${}^{\widetilde{G}}B_GU(n)$ reduces to a $G$-action.
If $V\subset \bf \widetilde{C}^\infty$ is a finite dimensional complex 
$\widetilde{G}$-subrepresentation, then there is a map
$${}^{\widetilde{G}}\gamma_GU(n) \oplus \IV \to {}^{\widetilde{G}}\gamma_GU(n+|V|)$$
which induces a map $\iota_V : {}^{\widetilde{G}}B_GU(n) 
\to {}^{\widetilde{G}}B_GU(n+|V|)$ at the level of the classifying spaces.
The colimit
\begin{align} \label{colimit twisted bundle}
{}^{\widetilde{G}}B_GU:= \colim_{V \subset \bf \widetilde{C}^\infty}
\bigsqcup_{n \geq 0 } {}^{\widetilde{G}}B_GU(n)
\end{align}
is the classifying space for reduced $\widetilde{G}$-twisted $G$-equivariant complex K-theory
$$ {}^{\widetilde{G}}\widetilde{K}^0_G(X) \cong [X, {}^{\widetilde{G}}B_GU]_G$$
for $X$ a finite $G$-CW complex.

Whenever the extension is trivial $\widetilde{G} \cong \IS^1 \times G$,
the spaces $ {}^{\IS^1 \times G}B_GU(n)$ classify
$G$-equivariant $U(n)$-principal bundles, and therefore we may denote 
$ B_GU(n):= {}^{\IS^1 \times G}B_GU(n)$ and
$B_GU:= {}^{\IS^1 \times G}B_GU$, thus having that for $X$ a compact $G$ space
$$\widetilde{K}^0_G(X)\cong [X,  B_GU]_G.$$
Furthermore, the
vector bundles $\gamma_GU(n):={}^{\IS^1 \times G}\gamma_GU(n)$ are the universal 
$G$-equivariant complex vector bundles of rank $n$.

Suppose now that $A$ is a closed subgroup of the compact Lie group $G$. Consider
the fixed point set $B_GU(n)^A$ and the restriction $\gamma_GU(n)|_{B_GU(n)^A}$
of the universal vector bundle. Denote by $N_A$ the normalizer of $A$ in $G$ and by
$W_A=N_A/A$ the quotient. Therefore $\gamma_GU(n)|_{B_GU(n)^A}\to B_GU(n)^{A}$ 
is a $N_{A}$-equivariant vector bundle such that $A$ acts trivially on the base 
space. In this way we are in the situation of the previous section for the 
short exact sequence $1\to A\to N_{A}\to W_{A}\to 1$. 
Take $\rho \in \Irr(A)$. By Theorem \ref{centralaction} the bundle
$\Hom_{A}(\IV_{\rho},\gamma_GU(n)|_{B_GU(n)^A})$ is a $(\widetilde{W_A})_\rho$-twisted
$(W_A)_\rho$-equivariant complex bundle, but since the space $B_GU(n)^A$ is not 
necessarily connected, it may not have constant rank. Therefore in order to construct a universal
$N_{A}$-equivariant complex bundle over spaces with trivial $A$-actions using universal
$(\widetilde{W_A})_\rho$-twisted $(W_A)_\rho$-equivariant complex bundles we need
to work with bundles of all ranks. We claim the following result:

\begin{theorem} 
\label{theorem decomposition of B_GU(n)}
There is a $W_A$-equivariant homotopy equivalence
$$\bigsqcup_{n=0}^{\infty}B_GU(n)^A \simeq 
\prod_{\rho \in W_A \backslash \Irr(A)} m^{W_A}_{(W_A)_\rho} 
\left( \bigsqcup_{n_\rho=0}^{\infty}{}^{(\widetilde{W_A})_\rho}B_{(W_A)_\rho}U(n_\rho) \right).$$
The stable version is
$$B_GU^A \simeq  \prod_{\rho \in W_A \backslash \Irr(A)} m^{W_A}_{(W_A)_\rho} 
\left({}^{(\widetilde{W_A})_\rho}BU_{(W_A)_\rho}U \right)$$
as $W_A$-spaces.
\end{theorem}

\begin{proof}
To start notice that by \cite[\S V, Lem. 4.7 \& \S VII, Thm. 2.4]{May-book} it follows that 
$B_GU(n)^A$ classifies $N_{A}$-equivariant complex vector bundles 
of rank $n$ over $A$-trivial $N_{A}$-spaces. Therefore $\bigsqcup_{n=0}^{\infty}B_GU(n)^A$ 
classifies $N_{A}$-equivariant complex bundles (of any rank) 
over spaces with trivial $A$-actions. The theorem will follow by Theorem 
\ref{theorem decomposition vector bundle} since both sides classify
$N_{A}$-equivariant complex bundles over spaces with trivial $A$-actions. Let us define the maps. 

Since $\Hom_{A}(\IV_{\rho},\gamma_GU(n)|_{B_GU(n)^A})$ is a $(\widetilde{W_A})_\rho$-twisted
$(W_A)_\rho$-equivariant complex bundle there is a $(W_A)_\rho$-equivariant classifying map
$$f_\rho:B_GU(n)^A \to \bigsqcup_{n_\rho=0}^{\infty}{}^{(\widetilde{W_A})_\rho}B_{(W_A)_\rho}U(n_\rho)$$ 
which induces a $W_A$-equivariant map
$$m^{W_A}_{(W_A)_\rho} (f_\rho) \circ p^{W_A}_{(W_A)_\rho} : B_GU(n)^A \to m^{W_A}_{(W_A)_\rho} 
\left( \bigsqcup_{n_\rho=0}^{\infty}{}^{(\widetilde{W_A})_\rho}B_{(W_A)_\rho}U(n_\rho) \right).$$
This constructs the map from left to right.

For the map from the right to the left, we know from Theorem \ref{centralaction} that
$$\bigsqcup_{n_\rho=0}^{\infty} \IV_{\rho} \otimes {}^{(\widetilde{W_A})_\rho}\gamma_{W_A}U(n_\rho)$$
is a $(N_A)_\rho$-equivariant complex bundle and 
$$ m^{N_A}_{(N_A)_\rho} \left(\bigsqcup_{n_\rho=0}^{\infty}\IV_{\rho} 
\otimes {}^{(\widetilde{W_A})_\rho}\gamma_{W_A}U(n_\rho) \right)$$
is a $N_A$-equivariant complex bundle. The product over $\rho \in W_A \backslash \Irr(A)$ is 
also a $N_A$-equivariant complex bundle and therefore there is a classifying map
$$\prod_{\rho \in W_A \backslash \Irr(A)} m^{W_A}_{(W_A)_\rho} 
\left( \bigsqcup_{n_\rho=0}^{\infty}{}^{(\widetilde{W_A})_\rho}B_{(W_A)_\rho}U(n_\rho) \right) \to
\bigsqcup_{n=0}^{\infty}B_GU(n)^A.$$
The homotopy equivalence follows from Theorem \ref{theorem decomposition vector bundle}. 
The homotopy equivalence of the stable version follows from 
Corollary \ref{corollary decomposition K-theory}.

\end{proof}

\begin{remark} \label{remark distribution} Distributing the product over union 
we obtain a homeomorphism 
\begin{align*}
\prod_{\rho \in W_A \backslash \Irr(A)} m^{W_A}_{(W_A)_\rho} 
\left(\bigsqcup_{n_\rho=0}^{\infty}  {}^{(\widetilde{W_A})_\rho}B_{(W_A)_\rho}U(n_\rho) \right) &\\
\cong 
\bigsqcup_{n=0}^\infty \bigsqcup_{\sum n_\rho |\rho|=n} \prod_{\rho \in \Irr(A)} 
& {}^{(\widetilde{W_A})_\rho}B_{(W_A)_\rho}U(n_\rho) 
\end{align*}
where $|\rho|$ denotes the complex dimension of the representation $\rho$. 
We note that the expression on right hand side is not canonically a $W_A$-space, 
therefore we induce the $W_A$-action on the  right hand side from the one of the 
expression on the left.

Let $\calp(n,A)$ be the set of arrays $P=(n_\rho)_{\rho \in  \Irr(A)}$ such that 
$$\sum_{\rho  \in \Irr(A)} n_\rho |\rho|=n$$ 
Restricting to $G$ equivariant complex bundles of rank $n$ over $G$-spaces 
one gets the $W_A$-homotopy equivalence
   $$B_GU(n)^A \simeq \bigsqcup_{{P \in \calp(n,A)   }} 
\prod_{\rho \in  \Irr(A)} 
\left({}^{(\widetilde{W_A})_\rho}B_{(W_A)_\rho}U(n_\rho) \right)$$
where $W_A$ acts on $\calp(n,A)$ permuting the arrays of numbers according to its action
on $\Irr(A)$, and 
the isotropy subgroups $({W_A})_\rho$ act on the appropriate coordinate space
${}^{(\widetilde{W_A})_\rho}B_{(W_A)_\rho}U(n_\rho)$.

Whenever $G$ is abelian $G$ acts trivially on $\Irr(A)$ and therefore we 
recover the homotopy equivalence of $G/A$-spaces
$$B_GU(n)^A \simeq \bigsqcup_{{P \in \calp(n,A)}} \prod_{\rho \in \Irr(A)} B_{G/A}U(n_\rho)$$
that appeared in \cite[ \S XXVI, Prop. 4.3]{May-book}. Note that in this case $\sum_\rho n_\rho=n$ since all irreducible 
representations of $A$ are 1-dimensional. 

Whenever the normalizer of $A$ is connected, the right hand
side simplifies as $$B_GU(n)^A \simeq \bigsqcup_{{P \in \calp(n,A) }} 
\prod_{\rho \in  \Irr(A)} {}^{\widetilde{W_A}_\rho}B_{W_A}U(n_\rho)$$
where $\widetilde{W_A}_\rho$ is the $\IS^1$-central extension of $W_A$ that depends on $\rho$.
\end{remark}

\section{Equivariant unitary bordism}\label{section bordism}

The decomposition of equivariant complex vector bundles on fixed point sets 
carried out in the previous sections is a key ingredient in the calculation of 
equivariant unitary bordisms for families. Conner and Floyd in their monumental 
work on the study of the bordism groups \cite{ConnerFloyd-book, ConnerFloyd-Odd} 
introduced the use of families of subgroups in order 
to restrict the bordisms to manifolds whose isotropy groups lie in a prescribed family.

Here we will concentrate in the tangentially stably equivariant unitary bordism groups 
$\Omega^G_*$, which will be called the geometric $G$-equivariant unitary bordisms. 
The explicit definition of these homology groups and their stable versions 
can be found in \cite[\S 2]{Hanke} and in \cite[  \S XXVI, Def. 3.1]{May-book}.  
Let us recall the main ingredients.  In all what follows $G$ will be a compact Lie group.

\begin{definition}
Let $M$ be a smooth $G$-manifold. A tangentially stably almost complex $G$-structure on $M$ 
is a complex $G$-structure on $TM \oplus \underline{\IR}^k$ for some $k\ge 0$, where 
$\underline{\IR}^k$ denotes the trivial bundle $M\times \IR^k$ over $M$ with trivial $G$-action, 
i.e. there exists a $G$-equivariant complex bundle $\xi$ over $M$ such that 
$TM \oplus \underline{\IR}^k \cong \xi$ as $G$-equivariant real vector bundles. Two tangentially 
stably almost complex $G$-structures are identified if after stabilization with further 
$G$-trivial $\underline{\IC}$ summands the structures become $G$-homotopic through complex 
$G$-structures. 
\end{definition}

With this definition if $H$ is a closed subgroup of $G$ then the fixed points 
$M^H$ also have a tangentially stably almost complex $N_{H}$-structure. Moreover,
a $N_{H}$-tubular neighborhood around $M^H$ in $M$ possesses a $N_{H}$-complex structure 
by \cite[\S XXVI, Prop. 3.2]{May-book}, let us see why. The bundle $ TM|_{M^H}$ contains  $T(M^H)$ as a subbundle and
$$
TM|_{M^H} = T(M^H) \oplus \nu(M^H,M)
$$
where  $\nu(M^H,M)$ is the normal bundle of $M^H$ in $M$. Also  $T(M^H)=(TM|_{M^H})^H$  
is a $N_H$-equivariant real vector bundle.

Given the tangentially stably almost complex $G$-structure $\xi$, we have that $\xi|_{M^H}$ is a complex vector bundle over $M^H$  with a complex vector subbundle $\xi^H$ also 
$$\xi^H \cong \left ( TM \oplus \underline{\IR}^k \right )^H =(TM|_{M^H})^H \oplus \underline{\IR}^k= T(M^H) \oplus \underline{\IR}^k,$$ 
therefore $M^H$ has a tangentially stably almost complex $N_H$-structure. 

Now, since 
$$\xi|_{M^H} \cong 
TM|_{M^H} \oplus \underline{\IR}^k = T(M^H) \oplus \nu(M^H,M)  \oplus \underline{\IR}^k$$
then the normal bundle $\nu(M^H,M)$  is isomorphic to the quotient bundle $\xi|_{M^H} /  \xi^H$, 
which is a $N_H$-complex bundle on $M^H$, thus showing that the normal bundle of  
$M^H$ on $M$ possesses a $N_H$-complex structure. 

\begin{definition}
For a cofibration of $G$-spaces $Y \to X$  the geometric $G$-equivariant unitary 
bordism groups $\Omega^G_n(X,Y)$ are defined as $G$-bordism classes of singular tangentially 
stably almost complex $n$-dimensional $G$-manifolds   $(M^n, \partial M^n) \to (X,Y)$. 
\end{definition}

When $G$ is trivial , a tangentially stably almost complex structure is the same as a normally stably almost complex structure and $\Omega^{\{1\}}_n(X,Y)$ is the usual unitary bordism of the pair $(X,Y)$.

One way to study the equivariant bordism groups is through the study of the equivariant 
bordism groups of manifolds $M$ whose isotropies lie in fixed family of subgroups of $G$.  
This way of studying equivariant bordism groups was developed by 
Conner and Floyd \cite[\S 5]{ConnerFloyd-Odd} and it is currently one of the most useful 
techniques to calculate the equivariant bordism groups.

A family of subgroups  $\calf$ of $G$ is a set (possibly empty) consisting of subgroups of $G$ which is 
closed under taking subgroups and under conjugation. Denote by $E\calf$ the 
classifying space for the family $E\calf$,  a $G$-space which is terminal in the 
category of $\calf$-numerable $G$-spaces \cite[\S 1, Thm 6.6]{tomDieck-transformation}, 
and which is characterized by the following properties on fixed point sets: 
$E\calf^H\simeq*$ if $H \in \calf$ and $E\calf^H = \emptyset$ if $H \notin \calf$. 

Given families of subgroups $\calf' \subset \calf$ of $G$ the induced map $E \calf' \to E\calf$ 
can be constructed so that it is a $G$-cofibration.
 
Following tom Dieck (\cite[p. 310]{tomDieck-Orbit-I}), we can define  
equivariant unitary bordism groups for families $\Omega_*^G[\calf,\calf']$ as follows. 
Given a $G$-space $\Omega_*^G[\calf,\calf'](X)$ is defined as 
$$\Omega_*^G[\calf,\calf'](X):= \Omega_*^G(X \times E\calf,X \times E\calf').$$

Alternatively, we may define the geometric $G$-equivariant unitary bordism groups 
$\Omega^G_n \{\calf ,\calf' \}(X,A)$ in a geometric way as was done in \cite[\S 2]{Stong-complex}. 
We recall the definition of the absolute unitary bordism groups $\Omega^G_n \{\calf ,\calf' \}(X)$ 
for completeness.

A $(\calf, \calf')$-{\it free geometric unitary bordism element} of $X$ is an equivalence class of 
a triple $(M,\partial M, f)$, where $M$ is an $n$-dimensional $G$-manifold endowed with 
tangentially stably almost complex $G$-structure which is moreover $\calf$-free, i.e. such that all 
isotropy groups $G_m=\{ g \in G ~|~ gm=m \}$ for $m \in M$ belong to $\calf$,  $\partial M$ is 
$\calf'$-free and  $f:M \to X$ is a $G$-equivariant map. Two triples  $(M,\partial M, f)$ and 
$(M',\partial M', f')$ are equivalent if there exists a $G$-manifold $V$ that is  
$\calf$-free such that $\partial V=M \cup M' \cup V^+$, and $M \cap V^+=\partial M$, 
$M' \cap V^+= \partial M'$, $M \cap M' = \emptyset$, $V^+ \cap(M \cup M') = \partial V^+$ 
 and $V^+$ is $\calf'$-free, together with a $G$-equivariant map 
$F:V\to X$ that restricts to $f$ on $M$ and to $f'$ on $M'$ .

\begin{definition}
The set of equivalence classes of $(\calf, \calf')$-free geometric unitary bordism elements of $X$, 
consisting of classes $(M,\partial M, f)$ where the dimension of $M$ is $n$, and under the operation 
of disjoint union, forms an abelian group denoted by $\Omega^G_n \{\calf ,\calf' \}(X).$ We refer  
to these groups as the geometric unitary bordisms of $X$ restricted to the pair of families 
$\calf' \subset \calf$. The equivalence class corresponding to the triple $(M,\partial M, f)$ will 
be denoted by $[M, \partial M, f]$.
\end{definition}

Notice that if $N$ is a  stably almost complex closed manifold, we can
define $[N] \cdot [M, \partial M, f] := [N \times M, N\times \partial M, f\circ \pi_M]$ thus making  
$\Omega^G_n \{\calf ,\calf' \}(X)$ a module over the unitary bordism ring $\Omega_*$.

The covariant functor $\Omega^G_* \{\calf ,\calf' \}$ defines a $G$-equivariant homology theory 
\cite[Prop. 2.1]{Stong-complex}. A natural transformation 
$\mu : \Omega^G_n \{\calf ,\calf' \}(X) \to \Omega^G_n [\calf ,\calf' ](X)$
can be defined as in  \cite[Satz 3]{tomDieck-Orbit-I} in the following way. 
Suppose that $(M,\partial M, f)$ is a representative of an element in 
$\Omega^G_n \{\calf ,\calf' \}(X)$. Since the inclusion $E\calf' \subset E\calf$ is $G$-cofibration, 
$\partial M$ is $\calf'$-free and $M$ is $\calf$-free, there is a $G$-equivariant map 
$k:M \to E \calf$ such that $k(\partial M)\subset E\calf'$. The $G$-equivariant map
\begin{align*}
(f,k):M &\to X \times E \calf\\
m &\mapsto (f(m),k(m))
\end{align*}
maps $\partial M$ into $X\times E \calf'$ and therefore $(f,k)$ becomes an 
element in $\Omega^G_n [\calf ,\calf' ](X)$.

\begin{proposition}
The natural transformation 
\begin{align*}
\mu : \Omega^G_n \{\calf ,\calf' \}(X) & \to \Omega^G_n [\calf ,\calf' ](X)\\
[M,\partial M,f]  &\mapsto [(M,\partial M, (f,k))]
\end{align*}
is an isomorphism.
\end{proposition}

The proof of this proposition follows the same lines as the one done by tom Dieck in 
\cite[Satz 3]{tomDieck-Orbit-I} in the case of equivariant unoriented bordism. We will 
not reproduce the proof here. 

The long exact sequence of the pair $(X \times  E \calf', X \times  E\calf)$ becomes
\begin{align*}
\cdots \longrightarrow \Omega^G_n \{\calf' \}(X)
 \longrightarrow
\Omega^G_n \{\calf\}(X) \longrightarrow
\Omega^G_n \{\calf ,\calf' \}(X) \longrightarrow
\Omega^G_{n-1} \{\calf' \}(X) \longrightarrow \cdots
\end{align*}

where  $\Omega^G_n \{\calf\}(X) = \Omega^G_n \{\calf ,\emptyset \}(X) $ is the bordism group of  $\calf$-free  tangentially stably almost complex closed manifold  with an equivariant map to $X$.  Note that for finite $G$, when $\calf=\{e\}$ then $\Omega^G_n \{\{1\}\}(X)$ is the bordism group of tangentially stably complex closed manifolds with a free $G$-action and an equivariant map to $X$ which can be identified with  the usual unitary bordism group $\Omega_n(EG \times_G X)$. 

Similarly for three families of representations $\mathcal{F}_1 \subseteq \mathcal{F}_2 \subseteq \mathcal{F}_3$, we have the corresponding long exact sequence of a triple
\begin{align*}
 \to \Omega^G_n \{\calf_2 ,\calf_1 \}(X)
  \to
\Omega^G_n \{\calf_3 ,\calf_1 \}(X)
  \to
\Omega^G_n \{\calf_3 ,\calf_2 \}(X)
 \to
\Omega^G_{n-1} \{\calf_2 ,\calf_1 \}(X)
  \to 
\end{align*}

Following the same argument 
as in \cite[Lemma 5.2]{ConnerFloyd-Odd} we can obtain the next lemma. 

\begin{lemma} \label{lemma equivalent bordism classes}
Let $(M^n,\partial M^n ,f)$ be a $(\calf, \calf')$-free geometric unitary bordism element of 
$X$ and $W^n$ a compact manifold with boundary regularly embedded in the interior of $M^n$ and 
invariant under the $G$-action. If $G_m \in \calf'$ for all $m \in M^n \backslash W^n$, then 
$[M^n,\partial M^n ,f]=[W^n, \partial W_n, f|_{W^n}]$ in $\Omega^G_n \{\calf ,\calf' \}(X)$.
\end{lemma}

A pair of families $\calf' \subset \calf$ of subgroups of $G$ is said to be an 
{\it adjacent pair of families of groups} if $\calf \backslash \calf' = (A)$, where $A$ is a 
subgroup of $G$ and  $(A)$ is the set of subgroups conjugate to $A$ in $G$.
We then say that $\calf $ and $\calf'$ differ by $A$. Notice that if $A$ 
is a normal subgroup of $G$, then a pair of families $\calf$ and $\calf$ 
differ by $A$ precisely if  $\calf= \calf' \sqcup \{A\}$. Moreover, in this case if 
$M$ is a $G$-manifold such that $G_{m}\in \calf$ for every $m\in M$ then the fixed point 
set $M^{A}$ has a free action of $G/A$.
 
Building on the notation of Theorem \ref{theorem decomposition of B_GU(n)} and 
Remark \ref{remark distribution} we denote by $\overline{\calp}(n,A)$ the set of arrays 
$\overline{P}=(n_\rho)_{\rho \in  \Irr(A), \rho \neq 1}$ of non-negative integers, where 
the number $n_1$ associated to the trivial representation is not considered, such that 
$$\sum_{\rho \in \Irr(A), \rho \neq 1} n_\rho |\rho| =n.$$
In the above equation $n_\rho$ is a non-negative integer and $|\rho|$ denotes the complex 
dimension of the representation $\rho$.  Suppose now that $A$ is a closed and normal subgroup 
of $G$. For any such partition $\overline{P}$ we define the space
$$B_{G/A}U(\overline{P}):= \prod_{\rho \in \Irr(A), \rho \neq 1} 
{}^{(\widetilde{G/A})_\rho}B_{(G/A)_\rho}U(n_\rho) $$
with the $G/A$-action induced by the homeomorphism shown in Remark \ref{remark distribution}.

As an application Theorem \ref{theorem decomposition of B_GU(n)} we obtain the following 
decomposition formula for the geometric $G$-equivariant unitary equivariant bordism groups. 
This decomposition is well known for the case of a compact abelian Lie group but is new 
for the case of non abelian groups. Since we follow the same line of argument as in 
the abelian case we only sketch part of its proof. 
 
\begin{theorem} \label{thm bordism for adjacent families}
Suppose that $G$ is a compact Lie group and let $A$ be a closed normal subgroup of $G$.
If $(\calf,\calf')$ is an adjacent pair of families of subgroups  of 
$G$ differing by $A$, then
$$\Omega_n^G\{\calf,\calf' \}(X) \cong \bigoplus_{0 \leq 2k \leq n -{\rm{dim}}(G/A) }
\Omega_{n -2k}^{G/A}\{\{1\}\}\left(X^A \times \bigsqcup_{\overline{P} 
\in \overline{\calp}(k,A)}B_{G/A}U(\overline{P})\right),$$
where $\{1\}$ is the family of subgroups of $G/A$ which only contains the trivial group. 
\end{theorem}

\begin{proof} We are going to define an isomorphism $\Phi$ between these groups. 
Suppose that  $f:M \to X$ is a $(\calf, \calf')$-free geometric unitary bordism of $X$ 
so that $[M, \partial M, f]\in \Omega_n^G\{\calf,\calf' \}(X)$. 
Let $M^A= M^A_1 \cup \cdots \cup M^A_l$ be a decomposition on disjoint manifolds, 
where each $M^A_j$ is an  $n_1(j)$-dimensional manifold which is moreover connected. 
By the $G$-equivariant tubular neighborhood, 
we may find pairwise disjoint tubular neighborhoods $U_j$ of $M_j^A$
in $M$ which are diffeomorphic to the $G$-manifolds $D(\nu_j)$ through the 
diffeomorphisms $\phi_j: U_j \stackrel{\cong}{\to} D(\nu_j)$,
where $D(\nu_j)$ denotes the unit disk  bundle
of the $G$-equivariant normal bundle $\nu_j \to M^A_j$ of the inclusion $M^A_j \subset M$ 
for $j=1,\dots, l$. By Lemma \ref{lemma equivalent bordism classes} we know that
$$[M, \partial M, f] = \sum_{j=1}^l[U_j, \partial U_j, f|_{U_j}] = 
\sum_{j=1}^l[D(\nu_j), S(\nu_j), f|_{U_j} \circ \phi_j^{-1}] $$
in $\Omega_n^G\{\calf,\calf' \}(X)$.
The bundle  $\nu_j\to M^A_j$ is a complex $G$-equivariant 
bundle with the  property that the trivial $A$-representation does not appear on the fibers. 
Let $r_{j}={\rm{rank}}_\IC ( \nu_j)$ so that $n=n_{1}^j+2r_{j}$ for each $j=1,\dots, l$ with $n_1^j$
the dimension of $M_j^A$.
By Theorem \ref{theorem decomposition of B_GU(n)} we know that
the bundle $\nu_j$ is classified by a $G/A$-equivariant map
$$\kappa_j : M^A_j \to \bigsqcup_{\overline{P} \in \overline{\calp}(r_{j},A)}B_{G/A}U(\overline{P}).$$
Now set $f_j:M^A_j \to X^A$ to be the restriction of $f$ to $M^A_j$,  
define the product map
$$(f_j, \kappa_j) : M^A_j \to X^A \times 
\bigsqcup_{\overline{P} \in \overline{\calp}(r_{j},A)}B_{G/A}U(\overline{P}).$$
Notice that $0\le {\rm{dim}}(G/A) \le n_{1}^{j}$  and thus 
the class $[M^A_j, \emptyset, (f_j, \kappa_j)]$ defines an element in 
$\Omega_{n-2r_{j}}^{G/A}\{\{1\}\}(X^A \times \bigsqcup_{\overline{P} \in 
\overline{\calp}(r_{j},A)}B_{G/A}U(\overline{P}))$ for $j=1,\dots, l$. 
We define 
$$ \Phi([M, \partial M, f]):= \sum_{j=1 }^l [M^A_j, \emptyset, (f_j, \kappa_j)].$$
We claim that map $\Phi$ is an isomorphism. To see that  $\Phi$  is surjective, suppose that 
$$[Y,\emptyset, \varphi: Y \to X^A \times \bigsqcup_{\overline{P} \in 
\overline{\calp}(k,A)}B_{G/A}U(\overline{P})]\in  
\Omega_{n-2k}^{G/A}\{\{1\}\}(X^A \times \bigsqcup_{\overline{P} \in 
\overline{\calp}(k,A)}B_{G/A}U(\overline{P})).$$ 
Let $p:E\to Y$ be the $G$-equivariant complex vector bundle defined by the map 
$\pi_2 \circ \varphi : Y \to \bigsqcup_{\overline{P} \in 
\overline{\calp}(k,A)}B_{G/A}U(\overline{P})$ as it is shown in 
Theorem \ref{theorem decomposition of B_GU(n)}. Since $Y$ is a closed manifold with a 
free $G/A$-action and with a tangentially stably almost complex $G/A$-structure,
then the closed unit disk of the bundle $D(E)$ is an $n$-dimensional manifold endowed 
with a tangentially stably complex $G$-structure. Moreover, the boundary $S(E)$ of $D(E)$ is 
$\calf'$-free since the trivial $A$-representation does not appear on the fibers of $E$. 
Denoting by $\psi: D(E) \to X$ the composition of the maps
$$D(E) \to Y \to X^A \hookrightarrow X,$$
where the first is the projection on the base, the second is $\pi_1 \circ \varphi$ 
and the third is the inclusion, we see that the bordism class
$[ D(E),S(E), {\psi}: D(E) \to X]$ lives in $\Omega_n^G\{\calf,\calf' \}(X)$ 
and by construction 
\[
\Phi([ D(E),S(E), {\psi}: D(E) \to X])=[Y,\emptyset, \varphi].
\]
This proves the surjectivity of $\Phi$. The injectivity of $\Phi$ can be proved 
in a similar way  as in the case of a compact abelian Lie group using 
Theorem \ref{theorem decomposition of B_GU(n)}.
\end{proof}

To be able to extend the previous theorem to general subgroups that are not necessarily normal, 
we consider $G=N_A$ and extend $N_A$-bordisms to $G$-bordisms with the change of groups formula \cite[\S XXVI, Lem. 3.4]{May-book}:
For any subgroup $H$  of a finite group $G$, we have an isomorphism
\begin{align*}
\Omega_*^H(X) & \stackrel{\cong}{\to} \Omega_*^G(G\times_HX)\\
[M, \partial M, f: M \to X]&  \mapsto 
[G \times_H M, \partial  (G\times_ H M) , G\times_H f : G \times_H M \to G \times_HX ].
\end{align*}

\begin{corollary} \label{cor bordism for adjacent families finite group}
If $(\calf,\calf')$ is an adjacent pair of families of subgroups of the finite group 
$G$ differing by the subgroup $A$,
then
$$\Omega_n^G\{\calf,\calf' \}(X) \cong \bigoplus_{0 \leq 2k \leq n }
\Omega_{n -2k}^{W_A}\{\{1\}\}(X^A \times \bigsqcup_{\overline{P} \in 
\overline{\calp}(k,A)}B_{W_A}U(\overline{P}))$$
where $\{1\}$ is the family of subgroups of $W_A$ which only contains the trivial group. 
\end{corollary}
\begin{proof} We just need to note that $M^A\cap M^{gAg^{-1}} = \emptyset$ whenever $g$ 
does not belong to $N_A$.
Therefore we can choose a $N_A$-equivariant tubular neighborhood $U$ of $M^A$ in $M$ such that
its $G$-orbit $G\cdot U$ is a $G$-equivariant tubular neighborhood of $G \cdot M^A$ and such that
$$G \times_{N_A} U \to G \cdot U, \ \ \ [(g,u)] \mapsto gu$$
is a $G$-equivariant diffeomorphism. Hence we have an isomorphism
\begin{align*}
\Omega^G_*\{\calf,\calf' \}(X) & 
\stackrel{\cong}{} \Omega^{N_A}_*\{\calf|_{N_A},\calf'|_{N_A} \}(X)\\
[M, \partial M , f ; M \to X]  & \mapsto [U, \partial U , f|_U : U \to X]
\end{align*}
which composed with the isomorphism of Theorem \ref{thm bordism for adjacent families} for 
the group $N_A$ and its normal subgroup $A$ provides the desired result.
\end{proof}

Let us use  the decomposition formula given in Corollary
\ref{cor bordism for adjacent families finite group}
 of the equivariant bordism groups for adjacent families in the case of finite groups to give an 
 alternative proof of Theorem 1.1 in \cite{Rowlett}. 
Let $(\calf,\calf')$ be an adjacent pair of families of subgroups of the finite group $G$ 
differing by the subgroup $A$, and consider the restriction map from the $G$-bordisms to 
$A$-bordisms. Let $\cala$ denote the family of all subgroups of $A$ and let $\calp$ denote 
the family of all subgroups of $A$ besides $A$ itself. The restriction
of $G$-manifolds to $A$-manifolds gives a homomorphism
$$r_A^G:\Omega_n^G\{\calf,\calf' \} \to\Omega_n^A\{\cala,\calp \}^{W_A} $$
which lies in the $W_A$-invariants since the action of the inner automorphisms of $G$ in 
$\Omega_n^G\{\calf,\calf' \}$ is trivial (see \cite[\S 20]{ConnerFloyd-book})
and the restriction map is $N_A$-equivariant.

Applying the fixed point construction done in Theorem \ref{thm bordism for adjacent families}
to both sides of the homomorphism above we obtain the following diagram with 
horizontal isomorphisms
$$\xymatrix{
\Omega_n^G\{\calf,\calf' \} \ar[r]^-{\cong} \ar[d]^{r_A^G} &
\bigoplus\limits_{0 \leq 2k \leq n } \Omega_{n -2k}^{W_A}\{\{1\}\}
\left(\bigsqcup\limits_{\overline{P} \in \overline{\calp}(k,A)}B_{W_A}U(\overline{P})\right) 
\ar[d]^{r^{W_A}_{\{1\}}}\\ 
\Omega_n^A\{\cala,\calp \}^{W_A} \ar[r]^-{\cong} & 
\bigoplus\limits_{0 \leq 2k \leq n  } 
\Omega_{n -2k}\left(\bigsqcup\limits_{\overline{P}' \in 
\overline{\calp}(k,A)}BU(\overline{P}')\right)^{W_A}.
}$$ 
Note that $B_{W_A}U(\overline{P})$ is a model for $BU(\overline{P})$ and therefore we may take
$BU(\overline{P}):=B_{W_A}U(\overline{P})$.

If we tensor with the ring $Z_P$ of $P$-local integers, where $P$ is the collection of primes
which do not divide the order of the group, the right vertical map 
$$\Omega_{n -2k}^{W_A}\{\{1\}\}\left( \bigsqcup_{\overline{P} \in 
\overline{\calp}(k,A)}B_{W_A}U(\overline{P})\right) \otimes Z_P \cong
\Omega_{n -2k}\left(\bigsqcup_{\overline{P} \in 
\overline{\calp}(k,A)}B_{W_A}U(\overline{P})\right)^{W_A} \otimes Z_P.$$
induces an isomorphism. Therefore the restriction map
$$r_A^G:\Omega_n^G\{\calf,\calf' \} \otimes Z_P \to\Omega_n^A\{\cala,\calp \}^{W_A} \otimes Z_P $$
becomes an isomorphism (cf. \cite[Prop. 3.1]{Rowlett}). The spaces $B_{W_A}U(\overline{P})$ 
are products of $BU(j)$'s and therefore the bordism groups $\Omega_*(B_{W_A}U(\overline{P}))$
are zero in odd degrees and $\Omega_*$-free in even degrees. Since 
$$\Omega_*({}^{(\widetilde{W_A})_\rho}B_{(W_A)_\rho}U(n_\rho) )$$
is $(W_A)_\rho$-invariant, then the action of $W_A$ on 
$$\bigoplus_{\overline{P} \in \overline{\calp}(k,A)} \Omega_{n-2k}(B_{W_A}U(\overline{P}))$$
permutes the generators and therefore the $W_A$ invariants are also $\Omega_*$-free. Hence we 
conclude that the bordism groups $\Omega_n^G\{\calf,\calf' \} \otimes Z_P$ for adjacent families 
are $\Omega_* \otimes Z_P$-free in even degrees and zero in odd degrees. Therefore the short exact
sequences $$0 \to \Omega^G_*\{\calf'  \} \otimes Z_P\to \Omega^G_*\{\calf  \}\otimes Z_P  
\to \Omega^G_*\{\calf, \calf'  \} \otimes Z_P \to 0$$ are all split for all pair of families of subgroups of $G$, $\Omega_*^G \otimes Z_P$ is
a $\Omega_* \otimes Z_P$-free module and there is a canonical isomorphism
$$\Omega^G_* \otimes Z_P \cong \bigoplus_{(A)} \Omega_*^A\{\cala,\calp \}^{W_A} \otimes Z_P$$
where $(A)$ runs over the set of conjugacy classes of subgroups of $G$ (cf. \cite[Thm. 1.1]{Rowlett}).

\section{Applications} \label{applications}
In this section we use Corollary \ref{cor bordism for adjacent families finite group} to calculate 
the $\Omega_*$-module structure of the equivariant unitary bordism groups of the dihedral groups of 
order $2p$, where $p$ is an odd prime number. 
 
Let $D_{2p}= \langle a,b | a^p=b^2=1, bab=a^{-1} \rangle $ denote the dihedral group 
of order $2p$. Notice that $\left<a\right>\cong \IZ/p$ is a normal 
subgroup of $D_{2p}$ and we have an extension of groups 
\[
1\to \IZ/p\stackrel{i}{\rightarrow} D_{2p}\stackrel{\pi}{\rightarrow}  \IZ/2\to 1.
\] 
If we write $\IZ/2=\{1,\tau\}$, where $\tau^{2}=1$, then the map $j(\tau)=b$ defines a splitting 
of the previous short exact sequence and thus $D_{2p}\cong \IZ/{p}\rtimes \IZ/2$. 

Let us recall what we know about the unitary bordism groups of free $D_{2p}$-actions.
Denote by $S^{2i-1}_-$ the sphere with the antipodal $\IZ/2$-action, and denote by 
$S^{2k-1}_{\lambda^l}$ the spheres with the action of $\IZ/p$ given by 
multiplying by $\lambda=e^{\frac{2 \pi l}{p}}$. 
By \cite[Cor. 2.5]{Kamata2} the bordism classes of the
$D_{2p}$-free unitary manifolds defined by the balanced products
$D_{2p} \times _{\IZ/2} S^{2i-1}_-$ and $D_{2p} \times _{\IZ/p} S^{4k-1}_{\lambda^l}$
with $\IZ/2 \cong \langle b \rangle $ and $\IZ/p \cong \langle a \rangle $
form a generating set of $\widetilde{\Omega}_*^{D_{2p}}\{\{1\}\}$ as a $\Omega_*$-module.
In \cite[Thm. 2.6]{Kamata2} it is shown that the map
\begin{align} \label{decomposition free D2p actions}
i_* \oplus j_*:(\widetilde{\Omega}_*^{\IZ/p}\{\{1\}\})^{\IZ/2} 
\oplus \widetilde{\Omega}_*^{\IZ/2}\{\{1\}\} \stackrel{\cong}{\to} 
\widetilde{\Omega}_*^{D_{2p}}\{\{1\}\}
\end{align}
induced by the balanced products is an isomorphism of $\Omega_*$-modules, where 
$$i_*\left (\frac{1}{2} \left ([S^{4k-1}_{\lambda^l}]+[S^{4k-1}_{\lambda^{p-l}}]\right )\right ) = 
[D_{2p} \times _{\IZ/p} S^{4k-1}_{\lambda^l}]  
\ \ \ \mbox{and} \ \ \ j_*[S^{2i-1}_-] = [D_{2p} \times _{\IZ/2} S^{2i-1}_-].$$
Here $\widetilde{\Omega}_*^G\{\{1\}\} := \widetilde{\Omega}_*(BG)$ denotes the reduced bordism 
groups of $BG$, i.e.   $\widetilde{\Omega}_*(BG)$ is the kernel of the augmentation map
$\Omega_*(BG) \to \Omega_*$.

Using \cite[Theorem 3]{Landweber-complex} we see that since  $H^{n}(BD_{2p};\IZ)=0$ for all $n\ge 1$ odd, then the projective dimension of  $\Omega_{*}(BD_{2p})$ over $\Omega_{*}$ is at most $1$. Since $\Omega_{*}(BD_{2p})$ 
contains torsion elements we conclude that it is not a $\Omega_{*}$ projective module and thus it has projective dimension $1$  over  $\Omega_{*}$.

For what follows we will use the notation 
$\Omega^{G}_{+}\{\calf\}(X):=\bigoplus_{n \text{ even}}\Omega^{G}_{n}\{\calf\}(X)$ and 
$\Omega^{G}_{-}\{\calf\}(X):=\bigoplus_{n \text{ odd}}\Omega^{G}_{n}\{\calf\}(X)$ for 
any family $\calf$ of subgroups of a  finite group $G$. Therefore 
$\Omega^{G}_{*}\{\calf\}(X)=\Omega^{G}_{+}\{\calf\}(X)\oplus \Omega^{G}_{-}\{\calf\}(X)$.  
Similarly for pairs of families. With this notation we have that 
$\Omega_+^{D_{2p}}\{\{1\}\} \cong \Omega_+$  and $\Omega_-^{D_{2p}}\{\{1\}\}$ is all torsion. 
Moreover, we can identify $\Omega_{*}$ with $\Omega_{+}$ as $\Omega_{-}=0$.

The following theorem was originally proved in \cite{Lazarov}
and we offer here a simpler proof which makes use of the results of the previous sections.

\begin{theorem} \label{thm D2p}
The unitary bordism group $\Omega_*^{D_{2p}}$ is a free $\Omega_*$-module on even dimensional 
generators.
\end{theorem}

\begin{proof}
Consider the families  $\calf_0 \subset \calf_1 \subset \calf_2 \subset \calf_3$  of subgroups 
of $D_{2p}$ defined as follows: 
\begin{align*}
\calf_0 : = & \{ \{ 1 \} \},\\
\calf_1 : = & \{ \{ 1 \}, \langle a \rangle  \},\\
\calf_2 : = & \{ \{ 1 \}, \langle a \rangle, \langle b \rangle, 
\langle aba^{-1} \rangle\dots,  \langle a^{p-1}ba^{1-p} \rangle  \}=all\setminus \{D_{2p}\},\\
\calf_3 := & all.
\end{align*}

The proof of the theorem will be based on the following facts that will be proved later:
\begin{itemize}
\item The unitary bordism group $\Omega^{D_{2p}}_{*}\{\calf_3,\calf_1\}$ is a free 
$\Omega_*$-module on even dimensional generators.
\item The unitary bordism group $\Omega^{D_{2p}}_{+}\{\calf_1, \calf_0\}$ is a 
free $\Omega_*$-module.
\item The boundary map $\delta:\Omega^{D_{2p}}_{+}\{\calf_3,\calf_1\} \to 
\Omega^{D_{2p}}_{-}\{\calf_1,\calf_0\}$ is surjective.

\item Both $\Omega^{D_{2p}}_{-}\{\calf_0\}$ and 
$\Omega^{D_{2p}}_{-}\{\calf_1,\calf_0\}$ have projective dimension 1 as  modules 
over $\Omega_{*}$.

\item The boundary map $\partial:\Omega^{D_{2p}}_{+}\{\calf_3,\calf_0\} \to 
\Omega^{D_{2p}}_{-}\{\calf_0\}$ is surjective.
\end{itemize}

Using these facts we can prove the theorem as follows. Since 
$\Omega^{D_{2p}}_{-}\{\calf_3,\calf_1\}$ is trivial, the long exact sequence 
associated to the families $\calf_0 \subset \calf_1 \subset \calf_3$ 
becomes
\begin{align*}
0 \to \Omega^{D_{2p}}_{+}\{\calf_1, \calf_0\} \stackrel{}{\rightarrow} 
\Omega^{D_{2p}}_{+}\{\calf_3,\calf_0\} 
\stackrel{\beta}{\rightarrow} \Omega^{D_{2p}}_{+}\{\calf_3,\calf_1\} 
\stackrel{\delta}{\rightarrow}& \\\Omega^{D_{2p}}_{-}\{\calf_1,\calf_0\} \to & 
\Omega^{D_{2p}}_{-}\{\calf_3,\calf_0\} \to 0.
\end{align*}
Since  $\delta:\Omega^{D_{2p}}_{+}\{\calf_3,\calf_1\} \to 
\Omega^{D_{2p}}_{-}\{\calf_1,\calf_0\}$ is surjective, the previous 
exact sequence yields the short exact sequence  
\[
0\to \text{Im}(\beta)\to \Omega^{D_{2p}}_{+}\{\calf_3,\calf_1\} \stackrel{\delta}{\rightarrow}
\Omega^{D_{2p}}_{-}\{\calf_1,\calf_0\}\to 0.
\]
We know that $\Omega^{D_{2p}}_{-}\{\calf_1,\calf_0\}$ has projective dimension  
1 as a module over $\Omega_{*}$. Since $\Omega^{D_{2p}}_{+}\{\calf_3,\calf_1\}$ 
is a free $\Omega_{*}$-module, we conclude by Schanuel's lemma,that $\text{Im}(\beta)$ must be a projective 
$\Omega_{*}$-module and hence free by \cite[Proposition 3.2]{Conner-Larry}. On the other 
hand, using the long exact sequence given above we obtain the 
short exact sequence 
\[
0 \to \Omega^{D_{2p}}_{+}\{\calf_1, \calf_0\} \stackrel{}{\rightarrow} 
\Omega^{D_{2p}}_{+}\{\calf_3,\calf_0\} 
\stackrel{\beta}{\rightarrow} \text{Im}(\beta)\to 0.
\]
As both $ \Omega^{D_{2p}}_{+}\{\calf_1, \calf_0\}$ and $\text{Im}(\beta)$  
are free modules over $\Omega_{*}$ we conclude that 
$\Omega^{D_{2p}}_{+}\{\calf_3,\calf_0\}$ is a free $\Omega_{*}$-module 
as well. Moreover, since the boundary map 
$\delta:\Omega^{D_{2p}}_{+}\{\calf_3,\calf_1\} \to \Omega^{D_{2p}}_{-}\{\calf_1,\calf_0\}$ 
is surjective, then $\Omega^{D_{2p}}_{-}\{\calf_3,\calf_0\}$ is trivial. 
Hence $\Omega^{D_{2p}}_{*}\{\calf_3,\calf_0\}$ is a free $\Omega_*$-module
on even dimensional generators.

Now, the long exact sequence associated to the families $\calf_0 \subset \calf_3$ becomes
\begin{equation*}
0\to \Omega_+  \stackrel{}{\rightarrow} \Omega_+^{D_{2p}} 
\stackrel{\gamma}{\rightarrow} \Omega_+^{D_{2p}}\{\calf_3, \calf_0\} 
\stackrel{\partial}{\rightarrow}\Omega_-^{D_{2p}}\{\calf_0\} \to \Omega_-^{D_{2p}} \to 0
\end{equation*}
since $\Omega_+^{D_{2p}}\{\calf_0\} \cong \Omega_+$ and $\Omega_-\{\calf_3, \calf_0\}=0$. 
We know that the boundary map $\partial:\Omega_+^{D_{2p}}\{\calf_3, \calf_0\} \to 
\Omega_-^{D_{2p}}\{\calf_0\}$ is surjective and thus we conclude that 
$\Omega_-^{D_{2p}}$ is zero. On the other hand, using the previous long exact sequence 
we obtain the short exact sequence
\[
0\to \text{Im}(\gamma)\to \Omega_+^{D_{2p}}\{\calf_3, \calf_0\} \stackrel{\partial}{\rightarrow}
\Omega_-^{D_{2p}}\{\calf_0\}\to 0.
\]
In this short exact sequence we know that $\Omega_-^{D_{2p}}\{\calf_0\}$ 
has projective dimension 1 as a $\Omega_{*}$-module and that
$\Omega_+^{D_{2p}}\{\calf_3, \calf_0\}$ is a free $\Omega_{*}$-module. By Schanuel's lemma we conclude  that $\text{Im}(\gamma)$ is a projective 
$\Omega_{*}$-module and hence free by \cite[Proposition 3.2]{Conner-Larry}.
Finally using the short exact sequence 
\[
0\to \Omega_+  \stackrel{}{\rightarrow} \Omega_+^{D_{2p}} 
\stackrel{\gamma}{\rightarrow} \text{Im}(\gamma) \to 0
\]
we conclude that $\Omega_+^{D_{2p}}$ is a free $\Omega_{*}$-module 
because $\Omega_+$ and $\text{Im}(\gamma)$ are free as well.  
The theorem follows.\end{proof}

Let us now check each one of the facts listed above.
\begin{lemma}
The unitary bordism group $\Omega^{D_{2p}}_{*}\{\calf_3,\calf_1\}$ is a free 
$\Omega_*$-module on even dimensional generators.
\end{lemma}
\begin{proof}
Consider the adjacent pairs of families $(\calf_3,\calf_2)$ and $(\calf_2,\calf_1)$ with $A=D_{2p}$ 
on the first case and $A=\langle b \rangle$ on the second. Since both $D_{2p}$ and 
$\langle b \rangle$ are their own normalizers in $D_{2p}$, then in both cases the group 
$W_A$ is trivial. By Corollary \ref{cor bordism for adjacent families finite group} we know that 
both $\Omega^{D_{2p}}_{*}\{\calf_3,\calf_2\}$ and $\Omega^{D_{2p}}_{*}\{\calf_2,\calf_1\}$
are isomorphic to unitary bordism groups of copies of $BU(k)$'s and therefore free 
$\Omega_*$-modules on even dimensional generators. 
The long exact sequence associated to the families 
$\calf_1 \subset \calf_2 \subset \calf_3$ implies that $\Omega^{D_{2p}}_{-}\{\calf_3,\calf_1\}$ 
is trivial and the short exact sequence 
$$0 \to \Omega^{D_{2p}}_{+}\{\calf_2,\calf_1\} \to \Omega^{D_{2p}}_{+}\{\calf_3,\calf_1\} 
\to \Omega^{D_{2p}}_{+}\{\calf_3,\calf_2\} \to 0$$
implies that the middle term is also a free $\Omega_*$-module.
\end{proof}

\begin{lemma}
The unitary bordism group $\Omega^{D_{2p}}_{+}\{\calf_1, \calf_0\}$ is a free $\Omega_*$-module.
\end{lemma}
\begin{proof}
By Corollary \ref{cor bordism for adjacent families finite group}  
we know that
\begin{equation*}
\Omega_*^{D_{2p}}\{\calf_1, \calf_0\} \cong 
\Omega_{*}^{\IZ/2}\{\{1\}\}
\left(\bigsqcup_{n_1,n_2,...,n_{p-1} \in \IN} 
BU(n_1)\times \cdots \times BU(n_{p-1})\right),
\end{equation*} 
where the number $n_l$ parametrizes the rank of the irreducible representation of $\IZ/p$ given by
multiplication of $e^{\frac{2 \pi l}{p}}$. The action of $\IZ/2$ interchanges the coordinates
\begin{align*}
BU(n_1)\times \cdots \times BU(n_{p-1}) &\to BU(n_{p-1})\times \cdots \times BU(n_1) \\
(x_1,...,x_{p-1}) & \mapsto (x_{p-1},...,x_1)
\end{align*}
and it only has fixed points whenever $n_l=n_{p-l}$ for all $1 \leq l \leq \frac{p-1}{2}$.  Therefore 
\[
\Omega_*^{D_{2p}}\{\calf_1, \calf_0\} \cong M_{*}\oplus N_{*},
\]
where $M_*$ is isomorphic to a direct sum of unitary bordism groups of copies of $BU(k)$'s (thus 
a free $\Omega_*$-module) and
\begin{align*}
N_{*}:=& \bigoplus_{n_1,n_2,...,n_{\frac{p-1}{2}} \in \IN} 
\Omega_{*}\left(E\IZ/2 \times_{\IZ/2} X^{2}\right),
\end{align*}
where $X:=\prod_{l_1}^{l=\frac{p-1}{2}} BU(n_l)$ and $\IZ/2$ acts on $X^{2}$ 
by permutation of the coordinates. Next we study 
$\Omega_*(E\IZ/2 \times_{\IZ/2} X^2)$.  
The second page of the Atiyah-Hirzebruch spectral sequence becomes
$$E^2_{s,t} \cong H_s(\IZ/2, \Omega_t(X^2))$$
and therefore $E^2_{s,odd}=0$ and $E^2_{2k,t}=0$ for $k>0$. Whenever $s=0$
we have that the groups $E^2_{0,*}$ are the $\IZ/2$-coinvariants $\Omega_*(X^2)_{\IZ/2}$. 
By \cite[Prop. 4.3.2 \& 4.3.3]{Kochman}
we know that $\Omega_*(BU(k))$ is a free $\Omega_*$-module  with basis 
$$ \{ \alpha_{j_1}\alpha_{j_2}\cdots \alpha_{j_s} | 1 \leq \alpha_{j_1} \leq \cdots
\alpha_{j_s}, \ s\leq k \}$$
where the degree of $\alpha_{j_1}\cdots \alpha_{j_s}$ is $2(j_1 + \cdots +j_s)$, therefore it follows
that the $\IZ/2$-coinvariants $\Omega_*(X^2)_{\IZ/2}$ is a free $\Omega_*$-module. 
Since the odd columns and the even rows are trivial, the vertical axis is a free $\Omega_*$-module
and the other components of the first quadrant are $\IZ/2$-torsion, then the spectral
sequence collapses on the second page. This implies
that $\Omega_+(E\IZ/2 \times_{\IZ/2} X^2)$ is isomorphic to the coinvariants 
$\Omega_*(X^2)_{\IZ/2}$, and therefore $\Omega_+(E\IZ/2 \times_{\IZ/2} X^2)$ is a free 
$\Omega_*$-module. Hence $N_+$ is a free $\Omega_*$-module, and therefore 
$\Omega_+^{D_{2p}}\{\calf_1, \calf_0\}$ is a free $\Omega_*$-module.
\end{proof}

\begin{lemma}
The boundary map $\delta:\Omega^{D_{2p}}_{+}\{\calf_3,\calf_1\} \to 
\Omega^{D_{2p}}_{-}\{\calf_1,\calf_0\}$ is surjective.
\end{lemma}
\begin{proof} Following the argument of the proof of the previous lemma it is enough to show 
that there are elements in $\Omega^{D_{2p}}_{+}\{\calf_3,\calf_1\}$
whose boundary correspond in $\Omega^{D_{2p}}_{-}\{\calf_1,\calf_0\}$ to the generators of  
$\Omega_-(E\IZ/2 \times_{\IZ/2} X^2)$ as $\Omega_*$-module
for $X= \prod_{l_1}^{l=\frac{p-1}{2}} BU(n_l)$.

The bordism group $\Omega_*(X)$ is generated as $\Omega_*$-module by unitary manifolds $M \to X$ 
(see \cite[Prop. 4.3.2 \& 4.3.3]{Kochman}) and therefore the trivial $\IZ [\IZ/2]$-submodule of 
$\Omega_*(X^2)$ is generated as a $\Omega_*$-module by the
manifolds $M^2 \to X^2$.  We claim that $\Omega_-(E\IZ/2 \times_{\IZ/2} X^2)$
is generated  as $\Omega_*$-module by the unitary manifolds 
\[
S^{2i-1}_-\times_{\IZ/2}M^2 \to E\IZ/2 \times_{\IZ/2} X^2.
\] 
This follows from the following argument.  Consider the maps 
$$S^{2i-1}_-\times_{\IZ/2}M^2  \to  E\IZ/2 \times_{\IZ/2} M^2 \to  E\IZ/2 \times_{\IZ/2} X^2$$
where the first one is induced by the inclusion $S^{2i-1}_-\to S^{\infty}_-=E \IZ/2$ and 
the second is induced by the map $M^2 \to X^2$. If the dimension of $M$ is $n$,
the composition of the maps  in homology
$$H_{2i+2n-1}(S^{2i-1}_-\times_{\IZ/2}M^2) \to H_{2i+2n-1}(E \IZ/2\times_{\IZ/2}M^2) 
\to H_{2i+2n-1}(E \IZ/2\times_{\IZ/2}X^2)$$
sends the volume form $[S^{2i-1}_-\times_{\IZ/2}M^2]$ to the $\IZ/2$-torsion class in  
the group 
$H_{2i+2n-1}(E \IZ/2\times_{\IZ/2}X^2)$ which corresponds in 
$E^2_{2i-1,2n} \cong H_{2i-1}(\IZ/2, H_{2n}(X^2))$ of the Serre spectral sequence
to the class in $H_{2i-1}(\IZ/2,  \IZ [M^2]) \cong \IZ/2.$
Therefore the homology classes in $H_{-}(E \IZ/2\times_{\IZ/2}X^2)$
defined by the volume forms of the unitary manifolds $S^{2i-1}_-\times_{\IZ/2}M^2$ 
generate the homology in odd degrees. 
This implies that the Thom homomorphism 
$\mu: \Omega_*(E \IZ/2\times_{\IZ/2}X^2) \to H_*(E \IZ/2\times_{\IZ/2}X^2)$ 
is surjective and that the bordism spectral sequence collapses. We conclude
that the unitary manifolds $S^{2i-1}_-\times_{\IZ/2}M^2 \to E\IZ/2 \times_{\IZ/2} X^2$ generate
$\Omega_-(E\IZ/2 \times_{\IZ/2} X^2)$ as $\Omega_*$-module.

Now consider the complex vector bundle $E \to M$ of rank $n_1+\dots+n_{(p-1)/2}$ 
that the map $M \to X = \prod_{l_1}^{l=\frac{p-1}{2}} BU(n_l)$ defines, with the 
appropriate induced action of $\IZ/p = \langle a \rangle$ on the fibers. Take the manifold 
$S^{2i-1}_-\times D(E \times E)$, where $D(E \times E)$ is the disk bundle of $E^2 \to M^2$ 
and define the $D_{2p}$ action on it as follows: for $(x,y,z) \in S^{2i-1}_-\times D(E \times E)$ 
let $a \cdot (x,y,z) := (x, ay, a^{-1}z)$ and $b\cdot (x,y,z):=(-x,z,y)$. The class of the 
$D_{2p}$-manifold $S^{2i-1}_-\times D(E \times E)$ lies in $\Omega^{D_{2p}}_{-}\{\calf_1,\calf_0\}$ 
and corresponds to the class of $S^{2i-1}_-\times_{\IZ/2}M^2 \to E\IZ/2 \times_{\IZ/2} X^2$ in 
$\Omega_*(E \IZ/2\times_{\IZ/2}X^2)$. The class of the $D_{2p}$-manifold 
$D(\IC^i_-)\times D(E \times E)$ lies in $\Omega^{D_{2p}}_{+}\{\calf_3,\calf_1\}$ and its boundary 
its the class of $S^{2i-1}_-\times D(E \times E)$ in $\Omega^{D_{2p}}_{-}\{\calf_1,\calf_0\}$. 
Hence the the boundary map $\Omega^{D_{2p}}_{+}\{\calf_3,\calf_1\} \to 
\Omega^{D_{2p}}_{-}\{\calf_1,\calf_0\}$ is surjective and the lemma follows.
\end{proof}

\begin{lemma}
As modules over $\Omega_{*}$, both $\Omega_{-}^{D_{2p}}\{\calf_{0}\}$ and 
$\Omega^{D_{2p}}_{-}\{\calf_1,\calf_0\}$ have projective dimension 1.
\end{lemma}
\begin{proof}
Notice that $\Omega^{D_{2p}}_{*}\{\calf_{0}\}=\Omega_{*}(BD_{2p})$  and thus it has projective dimension $1$ over $\Omega_{*}$.  On the other hand, 
using the previous lemma we conclude that the Thom map corresponding to 
$\Omega^{D_{2p}}_{-}\{\calf_1,\calf_0\}$ is surjective and thus  
\cite[Proposition 4]{Landweber-complex} implies that $\Omega^{D_{2p}}_{-}\{\calf_1,\calf_0\}$ 
has projective dimension at most 1 as a module over $\Omega_{*}$. The projective 
dimension of this module is 1 because it also contains torsion elements.
\end{proof}

\begin{lemma}
The boundary map $\partial:\Omega^{D_{2p}}_{+}\{\calf_3,\calf_0\} \to \Omega^{D_{2p}}_{-}\{\calf_0\}$ 
is surjective.
\end{lemma}
\begin{proof}
By the isomorphism described in formula \eqref{decomposition free D2p actions} we know that the
bordism classes $[D_{2p} \times _{\IZ/2} S^{2i-1}_-]$ and 
$[D_{2p} \times _{\IZ/p} S^{4k-1}_{\lambda^l}]$ generate $\Omega^{D_{2p}}_{-}\{\calf_0\}$ as a 
$\Omega_*$- module. Let $D(\IC^i_-)$ and $D(\IC^{2k}_{\lambda^l})$ denote the disks of the 
representations of $\IZ/2$ and $\IZ/p$ respectively whose boundary are $S^{2i-1}_-$ and 
$S^{4k-1}_{\lambda^l}$. The manifolds $D_{2p} \times _{\IZ/2} D(\IC^{i}_-)$ and 
$D_{2p} \times _{\IZ/p} D(\IC^{2k}_{\lambda^l})$ are both $(\calf_3, \calf_0)$-free, and 
their boundaries are $D_{2p} \times _{\IZ/2} S^{2i-1}_-$ 
and $D_{2p} \times _{\IZ/p} S^{4k-1}_{\lambda^l}$ respectively. Therefore the boundary map
$\partial:\Omega^{D_{2p}}_{+}\{\calf_3,\calf_0\} \to \Omega^{D_{2p}}_{-}\{\calf_0\}$ is surjective 
and the lemma follows.
\end{proof}

\bibliographystyle{abbrv} 
  \bibliography{Equivariant-K-theory-bibliography}
\end{document}